\documentclass[11pt,a4paper]{amsart}
 \usepackage[]{amsmath, amsthm, amsfonts, amssymb, amscd,mathtools}
\usepackage[mathscr]{eucal}
\usepackage{hyperref}
\usepackage[all]{xy}
\usepackage[usenames,dvipsnames]{color}
\usepackage{todonotes}\usepackage{rotating}
\usepackage[shortlabels]{enumitem}
\usepackage{tikz}
\usepackage[top=1.2in,bottom=1.2in,left=1.in,right=1.in]{geometry} 
\setlist[enumerate]{label=\it{(\roman*)},
  ref=\it{(\roman*)}}

\usepackage{dsfont}

\usepackage{blindtext}
\usepackage{subfiles}

\usepackage{stackengine}
\stackMath

\newcommand{\Betti}{\text{\rm B}}

\newcommand{\restr}[1]          {\vert_{#1}}
\newcommand{\ra}{\rightarrow}

\newcommand{\C}{\mathds{C}}

\renewcommand{\H}{\mathds{H}}

\renewcommand{\P}{\mathds{P}}
\newcommand{\Q}{\mathds{Q}}
\newcommand{\R}{\mathds{R}}

\newcommand{\Z}{\mathds{Z}}

\newcommand{\caB}{\mathcal{B}}
\newcommand{\caC}{\mathcal{C}}

\newcommand{\caH}{\mathcal{H}}

\newcommand{\caM}{\mathcal{M}}

\newcommand{\caO}{\mathcal{O}}

\newcommand{\caT}{\mathcal{T}}

\newcommand{\caV}{\mathcal{V}}
\newcommand{\caW}{\mathcal{W}}
\newcommand{\caX}{\mathcal{X}}

\newcommand{\caZ}{\mathcal{Z}}

\newcommand{\wt}{\widetilde}

\newcommand{\Pic}{{\rm Pic}}

\renewcommand{\hom}{\text{hom}}

\newcommand{\bfone}{{\mathds{1}}}
\newcommand{\fg}{{\mathfrak{g}}}

   \DeclareMathOperator {\Gr} {Gr}

 \DeclareMathOperator{\im}{Im}
 \DeclareMathOperator {\cl}{cl}
 \DeclareMathOperator {\CH}{CH}

\DeclareMathOperator{\Ext}{Ext}

\DeclareMathOperator{\Id}{Id}

\DeclareMathOperator{\GL}{GL}
\DeclareMathOperator{\gl}{\mathfrak{gl}}

\DeclareMathOperator{\Ad}{Ad}

\DeclareMathOperator{\Ht}{ht}

\numberwithin{equation}{section}

\theoremstyle{plain}
\newtheorem{prop}{Proposition}[section]

\newtheorem{cor}[prop]{Corollary}
\newtheorem{lem}[prop]{Lemma}
\newtheorem{thm}[prop]{Theorem}

\newtheorem*{mainthm}{Main Theorem}

\theoremstyle{definition}
\newtheorem{df}[prop]{Definition}

\theoremstyle{remark}
\newtheorem{rmk}[prop]{Remark}

\newtheorem{ex}[prop]{Example}

\begin{document}

\renewcommand\stackalignment{l}

\title{Genus three Ceresa cycles and limit of archimedean heights}

\author{Souvik~Goswami}

\address{Souvik Goswami, Institut des Hautes \'Etudes Scientifiques (IH\'ES), Le Bois-Marie, 35 Route de Chartres, 91440 Bures-sur-Yvette, France.} 
\email{gossouvik@gmail.com, goswamiS@ihes.fr}

\author{Irene Spelta}
\address{Irene Spelta, Institut f\"ur Mathematik, HU Berlin, Unter den Linden 6, 10099 Berlin,
Germany.}
\email{irene.spelta@hu-berlin.de}

\newif\ifprivate
\privatetrue
\subjclass[2020]{14C25, 14C30, 37P30}
\thanks{S.~Goswami was partially supported by the Universitat de Barcelona Mar\'ia Zambrano postdoctoral fellowship. I. Spelta is a member of GNSAGA (INdAM). Both authors were partially supported by the Spanish MINECO research project PID2023-147642NB-I00.}

\begin{abstract}
For a one-parameter variation of biextension mixed Hodge structures, Brosnan and Pearlstein showed that the limit of the asymptotic height of the variation is given by a certain \textit{limit height of the nilpotent orbit}. This limit height depends on the choice of a parameter. In the case of a variation of geometric origin related to Ceresa cycles associated with curves of genus three, after fixing a parameter, we show that this limit height is given by the Deligne splitting of a biextension mixed Hodge structure associated with cycles in the boundary.
\end{abstract}

\maketitle


\section*{Introduction}

A pair $(C,c)$, consisting of a smooth projective complex algebraic curve $C$ together with a marked point $c\in \Pic^{1}(C)$, determines an embedding $\iota_{c}\colon C\hookrightarrow JC,~ x\mapsto [x-c]\in JC$, and an algebraic $1$-cycle, known as the \emph{Ceresa cycle}. It is defined as $Z(C)\coloneqq C - C^{-}$, where $C^{-}$ denotes the image of $C$ under the $[-1]$ morphism inside $JC$. By construction, the Ceresa cycle is homologically trivial. However, a fundamental result of Ceresa (\cite{Ceresa}) shows that for a general curve of genus $g \geq 3$, it is not algebraically equivalent to zero. This nontriviality has made the Ceresa cycle a central object of study in the geometry of curves and their Jacobians. 

Let $\Delta\coloneqq \{t\in \C:~|t|<1\}$ be the open unit disc and $\Delta^{\ast}=\Delta\setminus \{0\}$. Consider a family $\caC\to \Delta$ of curves of fixed genus $g\geq 3$ with smooth total space. We assume that the general fibre $C_t,~t\neq 0$ is smooth and projective, while the central fibre $C_0$ has nodal singularities and is reduced. Two functions can naturally be attached to the family $JC_t$ of Jacobians associated with $C_t$. The first one is the normal function $\nu(t)\coloneqq AJ(Z(C_t))$ defined via the Abel-Jacobi image of the Ceresa cycle. Both $\nu$ and its derivative, the Griffiths infinitesimal invariant, have been studied widely by several authors (see, for example, \cite{Griff}, \cite{Voisin}, \cite{Hain:Height}, \cite{CollinoPirola}, \cite{KRP}, \cite{GGK:10}, and \cite{Hain:25}). The second function is the height, which we introduce more generally for a smooth complex projective variety $X$. Let $Z$ and $W$ be homologically trivial algebraic cycles of codimensions $p$ and $q$, respectively, intersecting properly, and satisfying $p + q=\dim X+1$. These assumptions imply that $Z$ and $W$ have disjoint support, and to the pair $(Z,W)$ we can associate a unique real number $\Ht(Z,W)$, called the  \emph {archimedean height}. Hodge theoretically, $\Ht(Z,W)$ measures how far a canonical subquotient $B_{Z,W}$ of the mixed Hodge structure $H^{2p-1}(X \setminus |Z|, |W|; \Q(p))$ is from being $\R$-split. $B_{Z,W}$ is an example of a \textit{biextension} mixed Hodge structure with graded pieces
\begin{displaymath}
\Gr^{W}_{0}B_{Z,W}=\Q(0),~ \Gr^{W}_{-1}B_{Z,W}=H^{2p-1}(X;\Q(p)),~\Gr^{W}_{-2}B_{Z,W}=\Q(1),
\end{displaymath}
and $\Gr^{W}_{k}B_{Z,W}=0$ otherwise. Conceptually, we can say that $B_{Z,W}$ `splices' together the Abel-Jacobi class $AJ(Z)$ and the dual $AJ(W)^{\vee}$ of the Abel-Jacobi class of $W$. One can define a metric on $B_{Z,W}$ by $||B_{Z,W}||\coloneqq e^{-\Ht(Z,W)}$. Similarly, for an abstract biextension $B$ (see section \ref{sec 3}) one can attach a height $\Ht(B)$, and define a metric $||B||=e^{-\Ht (B)}$.

Using the above definition of metric on biextensions and Ceresa cycles, Hain and Reed in \cite{HainReed} introduced a naturally metrized line bundle $(\mathcal{B}, \|\cdot\|_{\mathcal{B}})$ on the moduli space $\mathcal{M}_g$ of smooth projective curves of genus $g\geq 3$, called the biextension line bundle. For $g=3$ and for a point $[C]\in \caM_{3}$, this line bundle is heuristically given by the biextension mixed Hodge structure $B_{Z(C), W(C)}$ associated with a pair of Ceresa cycles $(Z(C), W(C))$ that intersect properly inside $JC$, and the associated metric is $||B_{Z(C),W(C)}||$.


A natural problem is to extend this construction from $\mathcal{M}_g$ to its Deligne--Mumford compactification $\overline{\mathcal{M}}_g$. It can be shown that while the biextension line bundle extends algebraically, the metric does not (see \cite{Lear:90}, \cite{HainReed}, \cite[Section 5]{pearlstein:sl2}). For an admissible variation of biextension type $\caB\to \Delta^{\ast}$ the height behaves as a rational multiple of $\log|t|$. Explicitly, one can show that the asymptotic function $\tilde{h}(t)\coloneqq \Ht(\caB_{t})+\mu\log|t|$ extends continuously to $t=0$, for a $\mu\in \Q$ that is controlled by the monodromy around $t=0$. This $\mu$ measures the singular behaviour of the height near the boundary. The above study was then generalized for an abstract biextension variation over a more general base in \cite{BP:Jump}. 

While the rational number $\mu$ has been studied extensively, we aim to study $\tilde h(0)$. In \cite[Definition 70]{BP:Jump}, the authors introduce a \textit{limit height $H(N,F_{\infty},W)$ of the nilpotent orbit} associated with an admissible biextension variation, such that (see \cite[Theorem 76]{BP:Jump})
\begin{displaymath}
\tilde{h}(0)=H(N,F_{\infty}, W).
\end{displaymath}
Naturally, the value of this limit height depends on the choice of a local coordinate. In this article, we aim to show that in a biextension variation of geometric origin, after fixing the choice of a coordinate for $\Delta$, $H(N,F_{\infty},W)$ is given by the Deligne splitting of a biextension mixed Hodge structure attached to the boundary variety. More precisely, following the approaches of \cite{Collino, Hain:Height, KRP}, we study the limit height of a family of Ceresa cycles associated to a special family of genus three curves, and obtain

\begin{mainthm}\label{mainthm}
Let $\caC\to \Delta$ be a family of genus $g=3$ projective curves with smooth total space. Let us assume that $C_t$ is smooth for $t\neq 0$, and that the central fiber $C_0$ is an integral nodal curve with a single node. Then the following happens:
\begin{enumerate}
\item\label{it-1}
There is a choice of coordinate $t$ of $\Delta$ such that the limit mixed Hodge structure $H^{1}_{\lim}(C_t;\R)$ is $\R$-split (see subsection \ref{subsec:2.2.1}).
\\
\item
There exist sections $\beta_i: \Delta \to \Pic^1(\mathcal{C}), i=1,2$ and disjoint families of Ceresa cycles $\{Z(C_t), W(C_t)\}_{t \in \Delta^*}$, constructed using $\beta_i(t)$ as base point, such that their limits $Z(C_0)$ and $W(C_0)$ are also disjoint in the compactification $X_0$ of the Jacobian $JC_0$, and are homologous to zero (see Propositions \ref{prop:4-1}, \ref{prop:4.5}).
\\
\item\label{it:3}
With respect to the choice of coordinate made in \ref{it-1}, the limit height $H(N,F_{\infty}, W)$ associated with the biextension variation $B_{Z(C_t),W(C_t)}$ is given by the Deligne splitting of the biextension mixed Hodge structure $B_{\wt Z(C_0),\wt W(C_0)}$, where $\wt Z (C_0)$ and $\wt W (C_0)$ are the lifts of $Z(C_0)$ and $W(C_0)$ to the normalization $\wt X_0$ of $X_0$ and are homologically trivial. Consequently, we get $H(N,F_{\infty}, W)=\Ht\Big(B_{\wt Z(C_0),\wt W(C_0)}\Big)$ (see Theorem \ref{thm:4-4} for the details). 
\end{enumerate}
\end{mainthm}

Our method is inspired by the work in \cite{KRP} and \cite{BP:Jump}, and can be thought of as a bridge between the two. Since $C_0$ is nodal and irreducible, there is a nontrivial vanishing cycle and a monodromy operator $N \neq 0$ with $N^2 = 0$ for $H^{1}(C_t;\Q)$. We fix a coordinate such that the limit mixed Hodge structure $H^{1}_{\lim}(C_t; \R)$ is $\R$-split (see subsection \ref{subsec:2.2.1}), and we study the limit of the biextension variation
\begin{displaymath}
\caV_{t}=H^{3}(JC_t \setminus |Z(C_t)|, |W(C_t)|; \mathbb{Q}(2))
\end{displaymath}
with respect to this fixed choice of the coordinate. The resulting limit mixed Hodge structure $\caV_{\lim}$ no longer has the structure of a biextension. Consequently, the analysis of the limiting height amounts to extracting from $\caV_{\lim}$ the graded pieces of the biextension $B_{\wt Z(C_0),\wt W(C_0)}=H^{3}(\wt X_0\setminus |\wt Z (C_0)|,|\wt W (C_0)|;\Q(2))$. This is the core motivation of our result (see Theorem \ref{thm:4-4}), and is achieved using a purity result for the cohomology of `mildly' singular varieties (see Proposition \ref{prop:4.4}), and comparing the limit Abel-Jacobi class $AJ_{lim}(Z_t)$ with the Abel-Jacobi class $AJ(Z_0)$, respectively $AJ_{lim}(W_t)$ with $W_0$. Finally, to show that $H(N,F_{\infty},W)$ is given by the Deligne splitting of $B_{\wt Z(C_0),\wt W(C_0)}$, we exploit a relation between heights and morphisms of mixed Hodge structures of a certain type (see Proposition \ref{height-ind-morphism}).

When the singular central fibre $C_0=D\cup E$ is reducible, uninodal, with $D$ and $E$ smooth projective curves of genera $1$ and $2$ respectively, the limit $H^{1}_{\lim}(C_t;\Q)$ is split for any choice of coordinates, and the corresponding limit mixed Hodge structure $\caV_{\lim}$ is a biextension mixed Hodge structure. Without fixing a coordinate, $\caV_{lim}$ and $B_{Z_0, W_0}$, the biextension associated with the cycles $Z_0$ and $W_0$ determined by the central fibre, are two biextension mixed Hodge structures that project to $(AJ(Z(C_0)), AJ(W(C_0)))$. One can make a choice of a coordinate $t$ such that the limit height equals the biextension height attached to the pair $(Z(C_0),W(C_0))$. However, this choice is not dictated by any splitting data of a limit mixed Hodge structure (unlike the integral nodal case), and is purely a numerical choice (see Theorem \ref{thm:3-8} for details). 

It is an interesting question whether one can choose a coordinate which is not merely a numerical adjustment, but is motivated by some properties of limit mixed Hodge structures. We will take this up, along with other examples of geometric origin, in a future project. In a slightly different direction, and when we start with a variation $\caX\to S$, where both $\caX$ and $S$ are smooth and defined over $\Q$, with $S$ being a curve with a basepoint $0\in S(\Q)$, one can also ask if the Bloch-Beilinson height pairing of the pair $(\wt Z_0, \wt W_0)$, now seen as elements in $\CH^{2}_{\hom}(\wt X_0/\Q)$, can be related to $H(N,F_{\infty}, W)$ associated with the complex variation $\caX_{\C}\to S_{\C}$, irrespective of the choice of a coordinate. This question was asked by S.~Bloch when the limit mixed Hodge structure is also a biextension, and has been answered positively for certain cases in \cite{Bloch at all.} and \cite{Bloch et all.-2}.\\

The paper is organized as follows. In Section \ref{sec 1}, we review basic material on mixed Hodge structures, including degenerations of mixed Hodge structures that have geometric origin, and review the theory of the Jacobian of a uni-nodal curve. In Section \ref{sec 3}, we discuss biextensions, both abstractly and in the geometric setting, introduce the notion of height, and recall the variational aspects of \cite{BP:Jump}. Finally, in Section \ref{sec 4}, we study the degeneration of families of Ceresa cycles, establish their properties and prove our main theorem (Theorem \ref{thm:4-4}).

\section{Preliminaries}\label{sec 1}
In this section, we gather preliminary results from Hodge theory, nodal curves, and Jacobians.
\subsection{Mixed Hodge structures}\label{sec:2}
In this section, we introduce rational mixed Hodge structures and list some properties (mostly without proofs) that we will need. We point out to the reader that working with rational mixed Hodge structures is a matter of convenience, more than anything else. One may as well work with integral structures modulo torsion, and the results of this article will remain valid.
\begin{df}
A rational mixed Hodge structure is a triple $$(H, W_{\bullet}, F^\bullet)$$ consisting of a $\Q$-vector space $H$, a finite increasing filtration $W_\bullet$ of $H$ (called the \textit{weight
filtration}) and a finite decreasing filtration $F^\bullet$ of $H_\C=H\otimes \C$ (called the \textit{Hodge filtration}) such
that for each $n\in \Z$ the couple $(\Gr_n^W H , \Gr_n^W(F^\bullet))$ is a pure $\Q$-Hodge structure of weight $n$. 
One can also consider real mixed Hodge structures, where instead
of a $\Q$-vector space $H_{\Q}$ one has a real vector space
$H_{\R}$. In fact, 
given a mixed $\Q$-Hodge
structure $H$, we will denote $H_{\R}=H_{\Q}\otimes \R$ obtaining an
$\R$-mixed Hodge structure.
\end{df}
When studying variations of mixed Hodge structures, it is convenient to
fix the underlying vector space and move the filtrations $F$ and $W$. Thus
if we fix an ($\R$ or $\Q$) vector space $V$, then a pair of
filtrations $(F,W)$ on $V\otimes \C$ and $V$ respectively is called a
mixed Hodge structure if the triple
\begin{displaymath}
  ((V,W),(V\otimes \C,W,F),\Id_{V\otimes \C})
\end{displaymath}
is a mixed Hodge structure.

\begin{ex}\label{ex:2.1}
For $a\in \Z$, the Tate mixed
Hodge structure $\Q(a)$ is the mixed Hodge structure given by the
following data
\begin{gather*}
  \Q(a)_{\Q}=\Q,\quad W_{-2a-1}\Q(a)_{\Q}=0,\quad
  W_{-2a}\Q(a)_{\Q}=\Q\\
  \Q(a)_{\C}=\C,\quad F^{-a}\Q(a)_{\C}=\C,\quad
  F^{-a+1}\Q(a)_{\C}=0
\end{gather*}
\end{ex}
\begin{df}
A morphism $f\colon (H, F, W)\to (H', F', W')$ of mixed Hodge structures is a $\Q$-linear map which is compatible with the Hodge and weight filtrations $(F,W)$ and $(F',W')$ respectively. For ease of notation, we will omit the Hodge and weight filtrations and denote a morphism of mixed Hodge structures simply by $f\colon H\to H'$.
\end{df}
\subsubsection{Deligne splitting}\label{ss:2.1}
A mixed Hodge structure $(F,W)$ on $V$ induces a unique functorial bigrading
\begin{equation}
    V_{\C } = \bigoplus_{a,b}\, I^{a,b}     \label{deligne-bigrading}
\end{equation}  
of the underlying complex vector space $V_{\C }$ such that
\begin{enumerate}
\item $F^a = \oplus_{\alpha\geq a,\beta}\, I^{\alpha,\beta}$;
\item $W_k = \oplus_{\alpha+\beta\leq k}\, I^{\alpha,\beta}$;
\item $\overline{I^{a,b}}
  \equiv I^{b,a} \mod\oplus_{\beta<b,\alpha<a}\, I^{\beta,\alpha}$.
\end{enumerate}
The $I^{a,b}$ is given by
\begin{equation*}
  \label{eq:72}
  I^{a,b}=F^{a}\cap W_{a+b}\cap \left(\overline{F^{b}}\cap W_{a+b} +
    \overline{U^{b-1}_{a+b-2}}\right), 
\end{equation*}
where
\begin{displaymath}
  U^{r}_{s}=\sum_{j\ge 0} F^{r-j}\cap W_{s-j}.
\end{displaymath}
A mixed Hodge structure is said to be \textit{split} (over $\R$) if $\bar{I}^{a,b}=I^{b,a}$. 
\begin{df}
The bigrading \eqref{deligne-bigrading} will be called the
\emph{Deligne bigrading} of $(F,W)$. The associated semisimple endomorphism
$Y = Y_{(F,W)}$ of $V_{\C }$ which acts as multiplication by $p+q$ on
$I^{p,q}$ will be called the \emph{Deligne grading} of $(F,W)$.
\end{df}
An important consequence of the Deligne bigrading is the strictness of a morphism of mixed Hodge structures with respect to the Hodge and the weight filtrations. Moreover, one also has that any short exact sequence of mixed Hodge structures
\begin{displaymath}
0\to H^{'}\to H\to H^{''}\to 0
\end{displaymath}
induces a short exact sequence on graded pieces
\begin{displaymath}
0\to \Gr^W_{i} H^{'}\to \Gr^W_{i} H\to \Gr^W_{i} H^{''}\to 0;~i\in \Z.
\end{displaymath}
Following Cattani, Kaplan, and Schmid (\cite{CKS:dhs}), let
\begin{equation*}
     \gl(V_{\C})^{a,b} = \{\,\alpha\in \gl(V_{\C})
        \mid \alpha(I^{c,d})\subseteq I^{a+c,b+d}\,\}
     \label{gl-hodge-comp}
\end{equation*}
be the Hodge decomposition of $\gl(V_{\C})$: an element $\lambda \in \gl(V_{\C})$ decomposes as $\lambda =\sum
\lambda ^{a,b}$.  Defining
\begin{equation}
     \Lambda^{-1,-1} = \bigoplus_{a<0,b<0}\, \gl(V_{\C })^{a,b}, 
     \label{lambda-def}
\end{equation}
then we have $\overline{\Lambda^{-1,-1}} = \Lambda^{-1,-1}$. It follows from the properties of the Deligne bigrading that taking $\lambda\in\Lambda^{-1,-1}$ then
   $(e^{\lambda}\cdot F,W)$ is another mixed Hodge structure on $V$ such that
   \begin{displaymath}
     I^{p,q}_{(e^{\lambda}\cdot F,W)} = e^{\lambda}(I^{p,q}_{(F,W)}).
   \end{displaymath} In particular, there exists a unique
real element $\delta = \delta_{(F,W)}\in\Lambda^{-1,-1}$ such that 
\begin{equation}
       \overline{Y}_{(F,W)} = e^{-2i\delta}\cdot Y_{(F,W)}  \label{delta-def}
\end{equation}
where $g\cdot \alpha \coloneqq \Ad(g)\alpha$ denotes the adjoint action of
$\GL(V_{\C })$
on $\gl(V_{\C })$. The element
$\delta$ defined by \eqref{delta-def}
will be called the \emph{Deligne delta} of $(F,W)$. It is responsible for turning a non-split mixed Hodge structure into a split one. 
\begin{rmk}\label{rmk:2.4}
Let $\delta_{-j}\colon \Gr^W_{i}V_{\C}\to \Gr^{W}_{i-j}V_{\C}$ be the component of $\delta$ that maps the $i$-th graded piece to the $i-j$-th graded piece, for all $i\in \Z$. Then
\begin{displaymath}
\text{ad}(Y)(\delta_{-j})(v)=(i-j)\delta_{-j}(v)-i\delta_{-j}(v)=-j\delta_{-j}(v).
\end{displaymath}
So $\delta_{-j}$ belongs to the $-j$-th eigenspace of $\text{ad}(Y)$, and we have a decomposition $\delta=\Sigma_{j}\delta_{-j}$ of $\delta$ in terms of these eigencomponents. Notice that, since $\delta\in \Lambda^{-1,-1}$, by \eqref{lambda-def}, $\delta_{-j}=\sum _{a+b=-j}\delta_{-j}^{a,b}$ where $a<0$ and $b<0$. Thus, the smallest possible shift is $j=2$.
\end{rmk}
Next, we show that $\delta$ commutes with morphisms of mixed Hodge structures.
\begin{lem}\label{lem:2.5}
Let $A$ and $B$ be mixed Hodge structures with Deligne
splittings $\delta_A$ and $\delta_B$ respectively.  Let $f:A\to B$
be a morphism of mixed Hodge structures.  Then,
$f\circ\delta_A = \delta_B\circ f$.
\end{lem}
\begin{proof}  By \cite[Proposition 2.20]{CKS:dhs}, if $C$ is a mixed Hodge structure,
then $\delta_C$ commutes with all $(r,r)$-morphisms of $C$.  Let
$C=A\oplus B$ and observe that $g(a,b) = (a,b+f(a))$ is a morphism of $C$.
Using the block structure of $\gl(C)=\gl(A\oplus B)$ it follows immediately
from \eqref{delta-def} that $\delta_C(a,b) = (\delta_A(a),\delta_B(b))$.
Writing out the $g\circ\delta_C = \delta_C\circ g$ shows that 
$f\circ\delta_A=\delta_B\circ f$.
\end{proof}

\subsection{Degeneration of mixed Hodge structures of geometric origin}\label{ss:3-1}
Let us set $\Delta\coloneqq \{z\in \C;~\vert z\vert <1\}$ and $\Delta^{\ast}\coloneqq \Delta\setminus \{0\}$. Further, let $f: \caX\rightarrow \Delta$ be a semistable degeneration with central fibre $X_0$ which is a complete simple normal crossing divisor. Let the generic fibre $X_t, t \in \Delta^{\ast}$ be smooth and of projective dimension $n$. By Deligne, once $X_0$ is a complete simple normal crossing variety, the $H:=H^*(X_0;\Q)$ can be equipped with a natural mixed Hodge structure coming from the Hodge structures provided by its normalization and by its irreducible components.  On the other hand, Schmid constructs a mixed Hodge structure on $X_0$ using the information encoded in its smoothing. We now recall the main facts (see \cite{schmid}).

Let $t_0$ be a smooth point, and $T: H^* (X_{t_0}; \Q)\ra H^* (X_{t_0}; \Q) $ the Picard-Lefschetz transformation of the family.  By Borel's theorem, $T$ is a quasi-unipotent operator. Let $0<m\in \Z$ be the smallest positive integer such that $T^{m}$ is unipotent. We define $N:=\frac{1}{m}\log T^m$ as the nilpotent logarithm of monodromy. By Landman's theorem (\cite{Lan:73}), we have $N^{n+1}=0$. 
Let $\Gamma$ be the monodromy group of the family, \begin{displaymath}
    \phi: \Delta^*\ra \Gamma \backslash D\quad \text{and   } \widetilde \phi: \caH\ra D
\end{displaymath}
be respectively the period map (over $\Delta^*$) and a lift to its universal covers. Then the function \begin{displaymath}
    \psi(z):= e^{-zN}F(z),
\end{displaymath}
satisfies the periodicity conditions \begin{displaymath}
    \psi(z+1)=\psi(z),
\end{displaymath}
and hence descends to a well-defined function $\psi(s)$ on $\Delta^*$ via the covering map $\caH\to \Delta^*, z\mapsto s=e^{2\pi iz}.$ Then, Schmid's Nilpotent orbit theorem assert that $\psi$ extends holomorphically over $0\in \Delta$, and the nilpotent orbit $\theta(z):=e^{zN}\cdot F_\infty$, with $F_\infty:=\psi(0)$, asymptotically  approximates $ \phi$. Moreover, the nilpotent orbit yields a limiting mixed Hodge structure as the data of the Hodge structure $$(H_\Q, F_\infty, W(N)),$$ where $W(N)_{*}$ is the  \textit{monodromy weight filtration} defined as the unique rational increasing filtration such that \begin{displaymath}
    N(W(N)_l)\subset W(N)_{l-2}, \forall , l\quad \text{and}\,  N^a: \Gr^{W(N)}_a\xrightarrow{\cong} \Gr^{W(N)}_{-a} , \forall\, a\geq 0.
\end{displaymath}
In other words, the family $\phi(t)$ of Hodge structure on $H^n(X_{t}, \Q)$ degenerates to the mixed Hodge structure $(F_\infty, W(N))$ as $t\ra 0$. Notice that $F_\infty$ depends on the choice of the coordinate, while $N, \theta(z)$ are independent of it.

In the case of a degeneration of mixed Hodge structure, the picture is more complicated since the limit mixed Hodge structure exists if the variation is \textit{admissible}, i.e., the Hodge filtration $F^\bullet_t$ extends, and the \textit{relative weight filtration} $M:=M(N, W)$ exists, where $M$ satisfies \begin{displaymath}
    N(M_l)\subset M_{l-2}, \forall , l\quad \text{and}\,~ N^l: \Gr^{M}_{l+k}\Gr^{W}_{k}\xrightarrow{\cong} \Gr^{M}_{k-l}\Gr^{W}_{k} , \forall\, k,l\geq 0.
\end{displaymath}
Remarkably, every geometric variation is always admissible. Finally, we recall that for such a limit MHS, the operator $N$ becomes a $(-1,-1)$-morphism. 

Deligne and Schmid's results are connected by the Clemens-Schmid exact sequence, which works as follows.
\begin{thm}\label{thm:3.1}
    Let $\caX\ra \Delta$ be a semistable degeneration. 
    \begin{enumerate}
        \item There exists a retraction $\caX\ra X_0$ which provides $H^n(\caX, \Q)\cong H^n(X_0, \Q) $ with a mixed Hodge structure. 
    \item There is the exact sequence of morphisms of mixed Hodge structures \[\ldots H_{2d+2-n}(\caX; \Q)\xrightarrow{\alpha}H^n(\caX; \Q)\xrightarrow{i^*}H^n(\caX_{t_0}; \Q)\xrightarrow{N}H^n(\caX_{t_0}; \Q)\xrightarrow{\beta}H_{2d-n}(\caX; \Q)\xrightarrow{\alpha}\ldots. \]
    \end{enumerate}
\end{thm}
As a consequence, cocycles invariant under monodromy, i.e. cocycles in $\ker N$, are cocycles on the special fibre. Furthermore, $N^{n+1}=0$ always and $N^n=0$ if and only if $H^n(\vert \Gamma(X_0)\vert )=0$, where $\vert \Gamma(X_0)\vert $ is the support of the dual graph of $X_0$.

\subsection{Nodal curves and Jacobians}\label{S1}

We now survey the structure of the Jacobian of a nodal curve and discuss the Abel-Jacobi maps. 
\subsubsection{Jacobian of a nodal curve}\label{ss-2-1}
Let $C_0$ be a connected irreducible nodal curve of genus $g$ with one node (ordinary double point) $x_{0}$ as a singularity. We define $JC_0$ as the variety of isomorphism classes of line bundles of degree zero on $C_0$, and the tensor product yields on $JC_0$ the structure of a group variety. Let $\pi\colon (\widetilde{C}_0, \{x,y\})\rightarrow (C_0, \{x_{0}\})$ be a normalization of $C_0$, where $\widetilde{C}_0$ is smooth projective curve of genus $g-1$ (opening the node), such that $\pi$ defines an isomorphism $\widetilde{C}_0\setminus \{x,y\}\cong C_0\setminus \{x_{0}\}$. Then, we define the pullback morphism
\begin{displaymath}
\pi^{\ast}\colon JC_0\rightarrow J\widetilde{C}_0=\Pic^{0}(\widetilde{C}_0).
\end{displaymath}
As it is well-known, this morphism is surjective, and the kernel is given by $\C^{\ast}$. 
Hence, we have an exact sequence which presents $JC_0$ as a semi-abelian variety:
\begin{displaymath}
0\rightarrow \C^{\ast}\rightarrow JC_0\xrightarrow{\pi^{\ast}} J\widetilde{C}_0\rightarrow 0.
\end{displaymath}
The semi-abelian variety $JC_0$ can be compactified. Indeed, identifying $\Pic^{0}(\widetilde{C}_{0})\cong \Pic^{0}(J\widetilde{C}_{0})$, the line bundle $L:=\caO_{\widetilde{C}_0}(x-y)$ can be viewed as a line bundle over $J\widetilde{C}_{0}$, and $JC_0$ is the total space of $L$ minus the 0-section. Following \cite{Altman Kleiman}, we give a quick description of the compactification: Let $Y_{0}$ and $Y_{\infty}$ be the zero and infinity sections of $\P(\caO_{J\widetilde{C}_{0}}\oplus L)$ obtained through the projections $s_{0}\colon \caO_{J\widetilde{C}_{0}}\oplus L\to\caO_{J\widetilde{C}_{0}}$ and $s_{1}\colon \caO_{J\widetilde{C}_{0}}\oplus L\to L$ respectively. Then, the compactification $X_{0}=\overline{JC_0}$ is obtained as \begin{displaymath}
    X_0=\P(\caO_{J\widetilde{C}_{0}}\oplus L)/Y_{0}\sim Y_{\infty}.
\end{displaymath} 
Notice that the two sections $Y_0$ and $Y_\infty$ do not get
identified fiberwise. Indeed, being $L$ an element in $J\widetilde C_0$, the identification works as \begin{equation}\label{identifiaction sections}
    (D, 0)  \in Y_0 \sim (D+x-y, \infty) \in Y_\infty.
\end{equation} 
As sets, we have: \begin{displaymath}
X_0:=\overline{JC_0}= JC_0\coprod J\widetilde C_0,
\end{displaymath}in other words the singular locus of $X_{0}$ is given by $J\widetilde{C}_{0}$, which is smooth and projective.   Elements in $X_0$ are rank one torsion-free sheaves $M$ of degree 0 whose sections $s$ are allowed to approach the point $x_0$.
The normalization $\widetilde{X}_{0}$ is given by the projective bundle $\P\left(\caO_{J\widetilde{C}_{0}}\oplus L\right)$ over $J\widetilde{C}_{0}$. From the projective bundle formula, we obtain an isomorphism of pure Hodge structures of weight $k$ for $H^{k}(\wt X_0;\Q)$
\begin{equation*}
H^{k}(\wt X_0;\Q)\cong H^{k}(J\wt C_0;\Q)\oplus H^{k-2}(J\wt C_0;\Q(-1)),
\end{equation*}
while $H^{k}(X_0;\Q)$ fits into the short exact sequence of mixed Hodge structures
\begin{equation}\label{eq:1.2}
0\to H^{k-1}(J\wt C_0;\Q)\to H^{k}(X_0;\Q)\to H^{k}(\wt X_0;\Q)\to 0.
\end{equation}
Using the exactness of the graded pieces, it is immediate that the weight graded $k$ component of $H^{k}(X_0;\Q)$ is given by $H^{k}(\wt X_0;\Q)$, and the $k-1$ component is given by $H^{k-1}(J\wt C_0;\Q)$.

\subsubsection{Abel-Jacobi map on stable curves}
It is well-known that, when $C$ is a smooth projective genus $g$ curve, fixing a base point $x$, the Abel Jacobi map is the embedding:
\begin{equation*}\label{AJacobi smooth}
    AJ: C\hookrightarrow JC, c\mapsto \mathcal{O}_C(c-x). 
\end{equation*}
In the last few decades, several authors have considered the problem of defining a similar map in the case of singular curves. The following is borrowed from \cite{Capo et all}. Let $C_0$ be a reduced, reducible, connected curve.
Furthermore, for each point $c\in C_0$, let $\mathcal{I}_c$ be its sheaf of ideals, which is torsion-free, rank 1, and of degree -1, and simple if $c\in C_0$ is a non-separating node. Then, if $C_0$ is free from separating nodes, and $x\in C_0$ is a smooth point, 
\begin{equation}\label{Abel non sep}
        AJ: C_0\rightarrow X_0,\;  c\mapsto \mathcal{I}_c\otimes \mathcal{O}_{C_0}(x)
    \end{equation}
    is the \textit{degree 0 Abel map} of $C_0$ with base point $x$. Furthermore, if $C_0\neq \mathbb{P}^1$, then \eqref{Abel non sep} is an embedding. When $c$ is a smooth point, $\mathcal{I}_c=\caO_{C_0}(-c)$, but, if $c$ is a node, then $\mathcal{I}_c$ is no longer locally free, and thus $AJ(c)$ is a boundary point in $\partial X_0$. 

On the other hand, if $C_0$ has a separating node $y_0$, the map \eqref{Abel non sep} has to be corrected, since multidegrees appear. In this case, $X_0$ parametrizes isomorphism classes of line bundles having degree 0 on every irreducible component. Furthermore, for each tail $Z$ of $C_0$, we consider the line bundles $\caO_{C_0}(Z)$ which restricts to $Z$ as $\caO_Z(-y_0)$, while on the complementary curve $Z'$ as $\caO_{Z'}(y_0)$. These are usually called \textit{twisters}. Then for $c\in C_0$ such that $c\neq y_0$, the sheaf $\mathcal{I}_c$ is, as before, the sheaf of ideals of $c$, while $\mathcal{I}_{y_0}$ is the unique (simple) line bundle such that
\begin{displaymath}
\mathcal{I}_{y_0}\restr Z\cong \caO_Z(-y_0) \quad \mathcal{I}_c\restr{Z'}\cong \caO_{Z'},
\end{displaymath}
where $Z$ is the $x$-tail not containing $x$, and $Z'$ the complementary. Hence, the \textit{degree 0 twisted Abel map} of $C_0$ with base point $x$ is defined as
\begin{equation}\label{Abel sep}
        AJ: C\rightarrow X_0,\;  c\mapsto \mathcal{I}_c\otimes \mathcal{O}_{C_0}(x)\otimes \caO_{C_0}(-\sum_{Z\in \caT_x(X)}Z),
    \end{equation}
where $\caT_x(X)$ is the set of all $x$-tails.

 \subsubsection{Degenerations}\label{degenerations} As is customary, let $\overline{\mathcal{M}}_g$ be the moduli space of stable genus $g$ algebraic curves. The complement $\overline{\mathcal{M}}_g\smallsetminus \caM_g$ is a divisor $D$ which is the union of $[g/2] +1$ irreducible components $D=D_0\cup D_1\cup \ldots\cup D_{[g/2]}$. The general point of $D_0$ is an irreducible
stable curve of genus $g$ with one node, while the general point of $D_i$ for $i > 0$ consists of a smooth curve of genus $i$ joined at one point to a smooth curve of genus $g- i$. In this article, we will focus on $g=3$ and look at the general element in the boundary divisors. The central fibre $C_0$ of our family can be of the following two types.
  \begin{enumerate}
      \item\label{itm-1.3.3-1} An irreducible curve $C_0$ with one node at $x_0\in C_{0}$ and normalization $(\widetilde{C}_0, \{x,y\})\to (C_0, \{x_0\})$.

      \item\label{itm-1.3.3-2} A reducible curve $C_0= C_1\cup C_2$ obtained as the union of two irreducible smooth components $C_j$ of genus $1$ and $2$ respectively, with a separating node $\{y_0\}=C_{1}\cap C_{2}$. Let $q_i \in C_i, i=1,2$ lying above $y_0$. In this case $X_0\cong JC_1\times JC_2.$
  \end{enumerate}
  In particular,  \eqref{Abel non sep} can be seen as $AJ: C_0\smallsetminus x_0\rightarrow JC_0, \; c\mapsto \caO_{C_0}(c-x)$ extended to the node via the compactification $X_0$. On the other hand, \eqref{Abel sep} works as follows:
    \begin{align}\label{cycle X}
        q\in C_1&\mapsto (\caO_{C_1}(q-q_1), \caO_{C_2}(q_2-x))\notag\\
        q\in C_2&\mapsto (\caO_{C_1}, \caO_{C_2}(q-x))\\
        y_0&\mapsto (\caO_{C_1}, \caO_{C_2}(q_2-x)).\notag
    \end{align}
    Thus, the image of $C_1$ in $JC_1$ is the Abel-Jacobi curve with base point $q_1$, while the image of $C_2$ in $JC_2$ is the Abel-Jacobi curve with base point $x$. Notice that, when the base point $x$ belongs to $C_1$, \eqref{Abel sep} reads as:
    \begin{align}\label{cycle X, punto base su C_1}
        q\in C_1&\mapsto (\caO_{C_1}(x-q), \caO_{C_2})\notag\\
        q\in C_2&\mapsto (\caO_{C_1}(x-q_1), \caO_{C_2}(q-q_2))\\
        y_0&\mapsto (\caO_{C_1}(x-q_1), \caO_{C_2}).\notag
    \end{align}

\section{Biextension and Heights}\label{sec 3}
Here, we introduce the mixed Hodge structures and associated heights, which will be the main focus of the article.
\subsection{Biextension mixed Hodge structures}
\begin{df}\label{df:2.1}
Let $H$ be a rational pure Hodge structure of weight $-1$. A \emph{biextension} $B_H$ associated with $H$ is a mixed Hodge structure that has at most three non
  trivial graded pieces, satisfying
\begin{displaymath}
  \Gr_{k}^{W}(B_H)=
  \begin{cases}
    \Q(0),&\text{ if }k=0,\\
    H,&\text{ if }k=-1,\\
    \Q(1),&\text{ if }k=-2,\\
    0,&\text{ otherwise.}
  \end{cases}
\end{displaymath}
\end{df}
The set $\caB_{H}$ of isomorphism classes of biextensions associated with $H$ is a $\C^{\ast}\otimes \Q$ bundle over $J(H)\times J(H)^{\vee}$, where $J(H)=\Ext^{1}(\Q(0), H)$ is the rational Jacobian associated to $H$ and $J(H)^{\vee}=J(H^{\vee})=\Ext^{1}(H^{\vee}, \Q(1))$ is its dual.
Since $J(H)\otimes\R=J(H)^{\vee}\otimes \R=0$, one can canonically identify (see \cite[Proposition 3.2.8, Corollary 3.2.9]{Hain:Height})
\begin{displaymath}
\caB_{H}\otimes \R\cong \Ext^{1}(\R(0),\R(1))\cong \R. 
\end{displaymath}
Hence, to any biextension $B_{H}$, one can attach a unique real number $\Ht(B_{H})=\{B_{H}\}\otimes \R$, called the (archimedean) \textit{height} of $B_{H}$, where $\{B_{H}\}$ denotes the isomorphism class. Notice that (cf. Remark \ref{rmk:2.4}) for a biextension mixed Hodge structure $\delta=\delta_{-2}=\delta^{-1,-1}$ with respect to $\text{ad}(Y)$, since there are no eigencomponents of degree $0$ or $-1$. Then it is easy to see the following.
\begin{thm}\label{height-bextn}
Given a biextension $B_{H}$ associated to a pure Hodge structure $H$ of weight $-1$, let $1_{B_{H}}\in \Gr^{W}_{0}B_{H}=\Q(0)$ and $1^{\vee}_{B_{H}}\in \Gr^{W}_{-2}B_{H}=\Q(1)$ be the canonical (Betti) generators. Let $e_{B_{H}}\in I^{0,0}B_{H,\C}$ be the unique element that projects to $1_{B_{H}}$ and $e^{\vee}_{B_{H}}\in I^{-1,-1}B_{H,\C}$ be that for $1^{\vee}_{B_{H}}$. Then
\begin{displaymath}
\delta_{-2}(e_{B_{H}})=\frac{1}{2\pi}\Ht(B_{H})e^{\vee}_{B_{H}}.
\end{displaymath}
\end{thm}
\begin{rmk}
From the above theorem, it is immediate that $\Ht(B_{H})=0$ if and only if $B_{H}$ is split over $\R$, that is $\overline{I}^{a,b}_{B_{H}}=I^{b,a}_{B_{H}}$ for all pairs $(a,b)\in \Z\times \Z$. In this sense, the height is an obstruction to $\R$-splitness of a biextension.
\end{rmk}

\subsubsection{Geometric biextension}\label{ss:2.2.1}
In geometric settings, biextension mixed Hodge structures arise as relative cohomology groups as follows. Let $X$ be a smooth, projective, and complex variety of dimension $d$, $Z$ and $W$ two algebraic cycles of codimensions $p$ and $q$ respectively that are homologous to zero and intersect properly. Let us also assume the numerical relation between the codimensions $p+q=d+1$ (together with the proper intersection, which means that $\vert Z\vert \cap \vert W\vert =\emptyset$). For ease of exposition, we also assume that $Z=Z_1-Z_0$ and $W=W_1-W_0$ for reduced and irreducible subvarieties $\{Z_{0},Z_{1},W_{0},W_{1}\}$. Let $H=H^{2p-1}(X;\Q(p))$, which is a pure Hodge structure of weight $-1$. Using the long exact sequences associated with cohomology with support and relative cohomology, together with purity, and the fact that $\vert Z\vert \cap \vert W\vert =\emptyset$, we get the commutative diagram
\ref{fig:height_diagram}. 
  \begin{figure}[ht]
    \centering
    \begin{displaymath}
      \xymatrix{& 0 & 0 & 0\\
        0\ar[r] & H\ar[u]\ar[r]
        & H^{2p-1}(X\setminus\vert Z\vert ;\Q(p))\ar[u]\ar[r]
        & \Q(0)\ar[u]\ar[r] & 0\\ 
        0\ar[r] & H^{2p-1}(X,\vert W\vert ;\Q(p))\ar[r] \ar[u]
        & H^{2p-1}(X\setminus\vert Z\vert ,\vert W\vert ;\Q(p)) \ar[r] \ar[u] & \Q(0)\ar[u]^{=}\ar[r] & 0\\
        0\ar[r] & \Q(1)\ar[r]^{=}\ar[u] & \Q(1)\ar[u]\\
        & 0\ar[u] & 0\ar[u]}     
    \end{displaymath}    
    \caption{Biextension diagram}
    \label{fig:height_diagram}
  \end{figure}
  
In the diagram, we have
\begin{displaymath}
\Q(0)=\ker\left(H^{2p}_{\vert Z\vert }(X;\Q(p))\to H^{2p}(X;\Q(p))\right),
\end{displaymath}
and
\begin{displaymath}
\Q(1)=\text{Coker}\left(H^{2p-2}(X;\Q(p))\to H^{2p-2}(\vert W\vert ;\Q(p))\right),
\end{displaymath}
noticing the relation $2p-2=2(d+1-q)-2=2\dim(W)$. The top horizontal exact sequence in diagram \ref{fig:height_diagram} gives the extension class for $Z$ and equals the Abel-Jacobi image of $Z$, while the left vertical exact sequence gives the dual of the extension class for $W$ twisted by $\Q(1)$.

From the functoriality of mixed Hodge structures, it is easy to see that $B_{Z,W}\coloneqq H^{2p-1}(X\setminus\vert Z\vert ,\vert W\vert ;\Q(p))$ satisfies the condition on the graded pieces to be a biextension. We remark that in case the cycles $Z$ and $W$ have more components, then one gets a subquotient of the mixed Hodge structure $H^{2p-1}(X\setminus\vert Z\vert ,\vert W\vert ;\Q(p))$ as a biextension.
\begin{rmk}\label{rmk:4-5}
Using the empty intersection between $\vert Z\vert $ and $\vert W\vert $ once again, we get the duality
\begin{displaymath}
H^{2p-1}(X\setminus\vert Z\vert ,\vert W\vert ;\Q(p))\cong H^{2q-1}(X\setminus\vert W\vert ,\vert Z\vert ;\Q(q))^{\vee}.
\end{displaymath}
This can be used to show that $B^{\vee}_{Z,W}\cong B_{W,Z}$, where $B_{W,Z}$ is the biextension obtained with the roles of $Z$ and $W$ interchanged.
\end{rmk}
\begin{rmk}\label{rmk:4.6}
In case of a biextension of geometric origin as above, the corresponding height $\Ht(B_{Z,W})$ equals the archimedean height pairing $\langle Z,W\rangle$ between the two cycles (this is a consequence of \cite[Proposition 3.2.13]{Hain:Height}. See also \cite[\S5]{pearlstein:sl2} for further insights.). In fact, using the duality of remark \ref{rmk:4-5} one can interpret $\Ht(B_{Z,W})$ as $\int_{Z}\mathfrak{g}_W$, where $\fg_{W}$ is a Green form of logarithmic type for the cycle $W$ that can be shown to exist. For example, in $\P^1$ any choice of four distinct points $\{p,q,r,s\}$ gives rise to homologically (actually rationally) trivial cycles $Z=p-q$ and $W=r-s$. In this case, one choice for $\fg_W$ is the function $-\log\Big|\frac{t-r}{t-s}\Big|^{2}$ that has logarithmic singularities on $|W|=\{r,s\}$. The height is given by the logarithm of the cross-ratio of the four points.
\end{rmk}

\subsection{Abstract degeneration of biextension variations}\label{ss:2.3}
Our main reference for this is \cite[\S 2, \S 3]{BP:Jump}. We point out that the treatment in \cite{BP:Jump} holds more generally for an admissible nilpotent orbit with values in the classifying space of graded-polarized mixed Hodge structures. Here we restrict to admissible, graded-polarized biextension variations over $\Delta^{\ast}$.
\begin{df}\label{df-bxtn}
Let $\caH$ be a weight $-1$ polarized variation of rational pure Hodge structures over $\Delta^{\ast}$ with unipotent monodromy. We define a biextension variation $\caB_{\caH}$ associated with $\caH$ as an admissible variation of mixed Hodge structures over $\Delta^{\ast}$, with isomorphisms
\begin{displaymath}
  \Gr_{k}^{W}(\caB_{\caH})=
  \begin{cases}
    \Q(0),&\text{ if }k=0,\\
    \caH,&\text{ if }k=-1,\\
    \Q(1),&\text{ if }k=-2,\\
    0,&\text{ otherwise.}
  \end{cases}
\end{displaymath}

\end{df}
For each $t\in \Delta^{\ast}$, one can associate the height $\Ht(\caB_{\caH,t})$, which defines a $C^{\infty}$-function
\begin{displaymath}
h\colon \Delta^{\ast}\to \R,~ h(t)\coloneqq \Ht(\caB_{\caH,t}).
\end{displaymath}
This function does not always extend continuously to $\Delta$. There is a $\mu\in \Q$, such that the amended function
\begin{displaymath}
\widetilde h(t)=h(t)+\mu\log\vert t\vert     
\end{displaymath}
has a continuous extension. For the rest of this section, we briefly discuss how one obtains this $\mu$ and a limit height. 

Let $\H$ be the upper-half plane and consider the covering map $z\mapsto t=e^{2\pi i z}$ from $\H$ to $\Delta^{\ast}$. In a neighbourhood of $t=0$, the biextension variation is of the form $(F(z), W),~z\in \H$. 
As in section \ref{ss:3-1}, let $N$ be the monodromy operator, $\caB_{\caH, \lim}=(F_{\infty},M)$ the limit mixed Hodge structure with relative weight filtration $M=M(N,W)$, and $(N,F_{\infty}, W)$ the admissible nilpotent orbit. Also, let $\delta_{M}$ be the Deligne splitting associated to the limit mixed Hodge structure $(F_{\infty}, M)$, and $\widetilde{F}_{\infty}=e^{-i\delta_{M}}F_{\infty}$ be the split-Hodge filtration. Then, there exists a unique functorial $\R$-grading of the original weight filtration $W$ defined first by Deligne (see \cite[Theorem 4.4]{GP01}, \cite[\S 3,\S4] {KP:03}, and \cite[Lemma 65]{BP:Jump} for details)
\begin{displaymath}
Y(N,F_{\infty},W)=\text{Ad}(e^{-iN})Y_{(e^{iN}\widetilde{F}_{\infty}, W)},
\end{displaymath}
which commutes with $Y_{(\widetilde{F}_{\infty}, M)}$. 
\begin{rmk}\label{rmk:4.8}
For admissible graded-polarized biextension variations, one has (\cite[Theorem 4.15]{pearlstein:sl2})
\begin{displaymath}
Y(N,F_{\infty}, W)=\lim_{\im(z)\to\infty}e^{-zN}\cdot Y_{(F(z),W)}.
\end{displaymath}
\end{rmk}
Since $Y(N,F_{\infty},W)$ is a grading of the original weight filtration, one has the decompositions of $N$ and $\delta_{M}$ into eigencomponents modulo $\text{ad}(Y(N,F_{\infty},W))$, with respect to the filtration $W$:
\begin{align*}
&N=N_{0}+N_{-1}+N_{-2},\\
&\delta_M=\delta_{M,0}+\delta_{M,-1}+\delta_{M,-2},
\end{align*}
where $\text{ad}(Y(N,F_{\infty},W))(N_{-j})=-jN_{-j}$, resp. $\text{ad}(Y(N,F_{\infty},W))(\delta_{M,-j})=-j\delta_{M,-j}$, for $j=0,1,2$ (since these are the weights for $W$). In fact, since we are working with a biextension variation, we have 
\begin{displaymath}
\Gr_{k}^{W}(\caB_{\caH,\lim})=
  \begin{cases}
    \Q(0),&\text{ if }k=0,\\
    \caH_{\lim},&\text{ if }k=-1,\\
    \Q(1),&\text{ if }k=-2,\\
    0,&\text{ otherwise},
  \end{cases}
\end{displaymath}
with respect to the original weight filtration $W$. Hence the components $N_{-2}$ and $\delta_{M,-2}$ can be extracted in a way that is made precise below.

Let $1$ denote the canonical Betti generator of $\Q(0)$ with the unique lift to
\begin{displaymath}
a\in \underbrace{I^{0,0}_{F_{\infty}, M}}_{Y(N,F_{\infty},W)~\text{weight}~0}
\end{displaymath}
and $1^{\vee}$ that of $\Q(1)$ with unique lift to
\begin{displaymath}
a^{\vee}\in \underbrace{I^{-1,-1}_{F_{\infty},M}.}_{Y(N,F_{\infty},W)~\text{weight}~ -2}
\end{displaymath}
Following \cite[Theorem 5.19]{pearlstein:sl2} and \cite[Theorem 68]{BP:Jump}, we define $\mu$ by
\begin{equation}\label{eq:3.1}
N_{-2}(a)=\mu a^{\vee}.
\end{equation}
Since $N$ does not depend on the choice of a coordinate for $\Delta$, neither does $\mu$. One still needs to show that $\mu\in \Q$, which is the content of \cite[Theorem 5.19]{pearlstein:sl2}. On the other hand, the limit height is defined by (\cite[(71)]{BP:Jump}):
\begin{equation}\label{eqref:3.2-1}
\delta_{M,-2}(a)=\frac{1}{2\pi}H(N,F_{\infty},W)a^{\vee}.    
\end{equation}
In this case, since the limit mixed Hodge structure depends on a choice of coordinate for $\Delta$, so does the height $H(N,F_{\infty}, W)$. The following result justifies the name `limit height' (\cite[Theorem 76]{BP:Jump}).
\begin{prop}\label{Prop:3.2-1}
The asymptotic height $\tilde{h}(t)=h(t)+\mu\log\vert t\vert $ extends continuously to $t=0$, with $\tilde{h}(0)=H(N,F_{\infty},W)$.
\end{prop}
\begin{rmk}\label{rmk:2.15}
Notice that $\caB_{\caH,\lim}$ may not have a biextension structure with respect to the relative weight filtration $M$. Also, the components $\delta_{M,0}$ and $\delta_{M,-1}$ can be more complex, depending on the degree of vanishing of the monodromy operator $N$.
\end{rmk}
Next we show the exact nature of the dependence of the limit height on the choice of coordinates:
\begin{prop}\label{Prop:3.2-2}
Let $Z_{\caB_{\caH,t}}\coloneqq \{(e^{\lambda N}.F_{\infty}, W);~\lambda\in \C\}$ be the nilpotent orbit corresponding to the biextension variation described above. Then
\begin{displaymath}
H(N, e^{\lambda N}.F_{\infty}, W)=H(N, F_{\infty}, W)+2\pi\Im(\lambda)\mu.
\end{displaymath}
\end{prop}
\begin{proof}
Let $\delta(F_{\infty})$ denote the Deligne delta for the mixed Hodge structure $(N, F_{\infty}, W)$ and $\delta(e^{\lambda N}.F_{\infty})$ that of $(N, e^{\lambda N}.F_{\infty}, W)$. Then, we have
\begin{displaymath}
\delta(e^{\lambda N}.F_{\infty})=\delta(F_{\infty})+\Im(\lambda)N.
\end{displaymath}
The corresponding Deligne grading satisfies
\begin{displaymath}
Y(N, e^{\lambda N}.F_{\infty},W)=e^{\lambda N}Y(N,F_{\infty}, W)e^{-\lambda N}.
\end{displaymath}
Since $e^{\lambda N}$ acts trivially on the graded pieces, taking the adjoint with respect to $Y(N, e^{\lambda N}.F_{\infty},W)$ and using the above relation gives
\begin{displaymath}
\delta(e^{\lambda N}.F_{\infty})_{-2}(a)=\delta(F_{\infty})_{-2}(1)+\Im(\lambda)N_{-2}(a).
\end{displaymath}
Finally, multiplying by $2\pi$ on both sides, and noting that $N_{-2}(a)=\mu a^{\vee}$, we get the desired relation.
\end{proof}
In the next section (\ref{subsec:2.2.1}), we will make such a choice of coordinates. Next, we have an easy corollary, the purpose of which will be relevant later.
\begin{cor}\label{cor:3.2-3}
Suppose we fix $\Im(\lambda)$, and vary the real part. Then we get a real sub-orbit of the Nilpotent orbit along which the limit height is constant.
\end{cor}

We end with a remark on the connection between the coordinates $z\in \H$, $t\in \Delta^{\ast}$, and $\lambda\in \C$.
\begin{rmk}\label{rmk:on the change of coordinate}
First, as we have seen earlier, $t=e^{2\pi iz}\in \Delta^{\ast}$ is given by the covering map. Next, we focus on $\lambda\in \C$. When we write $e^{\lambda N}.F_{\infty}$, it should really be $e^{\lambda N}.F_{\infty,t}$, since the limit MHS $(F_{\infty}, M)$ is really with respect to a choice of coordinate $t$. If we write $t=e^{2\pi i z_t}$, where the index $z_t$ is to remember the coordinate $t\in \Delta^{\ast}$, we should have $e^{\lambda N}.F_{\infty}=F_{\infty,s}$, for some other choice of coordinate $s=e^{2\pi i z_s}$. Now, up to the first order (which is the only relevant order for limit mixed Hodge structures), we have $z_s=z_t-\lambda$. Writing $\lambda=-\frac{1}{2\pi i}\log a$, for a suitable branch of logarithm, we get $\Im(\lambda)=\frac{1}{2\pi }\log|a|$. Therefore, the change of coordinate $s=at$, yields
\begin{displaymath}
H(N, F_{\infty, s}, W)=H(N, F_{\infty, t}, W)+\mu\log|a|,
\end{displaymath}
that aligns more closely with the description in \cite{BP:Jump}.
\end{rmk}

\subsection{Heights of framed mixed Hodge structures}
The notion of heights associated with biextension mixed Hodge structures has been generalized first in \cite{BGGP:Height} to oriented mixed Hodge structures and then in \cite{BGGP:Height2} to framed mixed Hodge structures. We briefly discuss this generalized notion of height, since the type of limit mixed Hodge structures we will encounter later will not be biextensions, but one can still define this generalized height on them. For details, the reader can consult \cite[\S2.2,\S2.3]{BGGP:Height2}.
\begin{df}\label{intro-def-1}
Let $\Q(n),~n\in \Z$ be the Tate mixed Hodge structure of pure type,
with Betti generator $\bfone(n)_{\Betti}\coloneqq 1\in \Q$. For $(a,b)\in \Z^{2}$,  an
$(a,b)$-\textit{framing} on a mixed Hodge structure $H$ is a pair 
of morphisms
\begin{displaymath}
  \phi\colon
  \Q(-a)\rightarrow \Gr^{W}_{2a}(H)\text{ and }
  \psi\colon \Gr^{W}_{2b}(H)\rightarrow \Q(-b)
\end{displaymath}
of pure Hodge structures.
\end{df}
To emphasize a framing on $H$, we sometimes denote it by
$H_{a,b}$. Let $e_{H}$ be the unique lift in $I^{a,a}_{H}$
($(a,a)$-component of the Deligne bigrading of $H$) of the Betti
generator $\bfone(-a)_{\Betti}\in \Q(-a)$. It is easy to show that if $H$ has a
framing given as above, then its dual $H^{\vee}$ also has a framing
given by morphisms
\begin{displaymath}
  \phi^{\vee}\colon
  \Q(b)\rightarrow \Gr^{W}_{-2b}(H^{\vee})\text{ and }
  \psi^{\vee}\colon \Gr^{W}_{-2a}(H^{\vee})\rightarrow \Q(a).
\end{displaymath}
Let $e_{H^{\vee}}$ be the element equivalent to $e_{H}$, inside
$I^{-b,-b}_{H^{\vee}}$. With this setup, we define two notions of
height: 
\begin{displaymath}
\Ht_{1}(H)\coloneqq\Im\langle e_{H^{\vee}}, \overline{e_{H}}\rangle
\end{displaymath}
and
\begin{displaymath}
\Ht_{2}(H)\coloneqq \langle e_{H^{\vee}}, \delta_{H}(e_{H})\rangle.
\end{displaymath}
Here $\langle~,~\rangle$ denotes the duality between $H$ and
$H^{\vee}$ and $\delta_{H}$ is the Deligne splitting of $H$ (see
\eqref{delta-def}). It can be shown that $\Ht_{2}(H)$ is already real, while for $\Ht_{1}(H)$ we take the imaginary
component. In general, these heights are not the same. However, if $\delta^{3}_{H}(e_H)=0$, we have $\Ht_{2}(H)=-\frac{1}{2}\Ht_{1}(H)$. Further, one has the following result:
\begin{prop}\label{height-ind-morphism}
Let $f\colon (H, \phi, \psi)\rightarrow (H', \phi', \psi')$ be a framed morphism of $(a,b)$-framed
mixed Hodge structures, which means $\phi'=\Gr_{2a}^{W}(f)\circ \phi$ and
$\psi=\psi'\circ \Gr_{2b}^{W}(f)$. Then for $i=1,2$,
$\Ht_{i}(H)=\Ht_{i}(H')$. 
\end{prop}

\section{Degeneration of curves, Ceresa cycles, and their heights}\label{sec 4}
In this section, given a family of $\caC\to \Delta$ whose central fibre is singular as in subsection \ref{degenerations}, we survey on the mixed Hodge structure we obtain in the limit. Then, we fix the genus at $g=3$ to maintain parity with our main discussion,  and we introduce (a family of) Ceresa cycles and describe its behaviour in the limit. 

\subsection{Degeneration of smooth projective curves and MHS}
Given a degeneration as in subsection \ref{degenerations}, we now study the limit mixed Hodge structure $H^{k}_{lim}(JC_t;\Q)$. For the sake of completeness, we start from the structure of $H^{1}_{lim}(C_t;\Q)$, even if it is well-known in the literature (for example, see \cite{Carlson}).

\subsubsection{Central irreducible uninodal curve}\label{subsec:2.2.1}
We consider the first degeneration described in subsection \ref{degenerations}: Here, the central fibre over $0\in \Delta$ is an irreducible nodal curve $C_0$, the node $x_0$ is obtained by identifying two points $x,y$ in the normalization $\widetilde C_0$. The natural variation of Hodge structure, whose fibre over $t\in \Delta$ is provided by $H^1(C_t, \mathbb{Q})$, has a limit MHS, as we now recall. Let $H_1(C_t, \mathbb{Z})=\langle\alpha_1,\alpha_2, \alpha_3, \beta_1, \beta_2, \beta_3\rangle$ be a symplectic basis (similarly $\langle\alpha^*_1,\alpha^*_2, \alpha^*_3, \beta^*_1, \beta^*_2, \beta^*_3\rangle$ the dual basis in cohomology), and $\alpha_3$ the vanishing cycle of the degeneration. Then, since $N^2=0$, the monodromy weight filtration has 3 weights and 
    \begin{gather*}
        M_0(H^1_{lim}(C_t; \Q))=\operatorname{Im} N\\
        M_1(H^1_{lim}(C_t; \Q))=\operatorname{ker} N=H^1(C_0; \Q)\\
        M_2(H^1_{lim}(C_t; \Q))=H^1_{lim}, 
    \end{gather*}
where the equality in $M_1$ follows from Theorem \ref{thm:3.1}, since up to a base change, our family admits a semistable degeneration. We  have:
\begin{displaymath}
\operatorname{Gr}^M_0H^1_{lim}(C_t; \Q)=\Q\beta_g^*,\;  V=\operatorname{Gr}^M_1H^1_{lim}(C_t; \Q)=H^1(\wt C_0; \Q),\; \text{and} \operatorname{Gr}^M_2H^1_{lim}(C_t; \Q)=\Q(-1)\alpha_g^*.
\end{displaymath}
The extension of $\operatorname{Gr}^M_1$ by $\operatorname{Gr}^M_0$ belongs to $\operatorname{Ext^1}( \mathbb{Q}, H^1(\widetilde C_0;\Q))\cong J\widetilde C_0\otimes \Q$, and it is given by the class of $x-y$. Moreover, since both $\operatorname{Gr}^M_0H^1_{lim}(C_t; \Q)$ and $\operatorname{Gr}^M_2H^1_{lim}(C_t; \Q)$ are one dimensional, the $\delta_{M}$ and $N$ are proportional. Let $\delta_{M}=cN,~c\in \R$. Following Proposition \ref{Prop:3.2-2}, we can write $\delta(e^{\lambda N}.F_{\infty})=\delta_{M}+\Im(\lambda)N=(c+\Im(\lambda))N$. Thus, we can choose $\lambda$ such that $\Im(\lambda)=-c$ and hence a coordinate $t$ such that $H^{1}_{lim}(C_t; \R)$ is split. From now on, we fix this coordinate once and for all.

Now, using the fact that $H^{k}(JC_t; \Q)\cong \wedge^{k} H^{1}(C_t; \Q)$, we compute $H^k_{lim}(JC_t: \mathbb{Q})$, by taking the exterior power $\wedge^k H^1_{lim}(JC_t; \mathbb{Q})\cong \wedge^k H^1_{lim}(C_t; \mathbb{Q})$. Moreover, we keep calling the induced monodromy operator $N$.
It follows that the limit MHS $H^k_{lim}(JC_t; \mathbb{Q})$ has 3 weights with respect to relative weight filtration $M$, with graded pieces:
  \begin{gather*}
        \operatorname{Gr}^M_{k-1}H^k_{lim}(JC_t; \mathbb{Q})= H^{k-1}(J\widetilde C_0,\Q ),\\
         \operatorname{Gr}^M_{k}H^k_{lim}(JC_t; \mathbb{Q})= H^{k-2}(J\widetilde C_0;\Q(-1) )\oplus H^{k}(J\widetilde C_0;\Q))\cong H^k(\wt X_0; \Q),\\
         \operatorname{Gr}^M_{k+1}H^k_{lim}(JC_t; \mathbb{Q})= H^{k-1}(J\widetilde C_0;\Q(-1)).
    \end{gather*}
The isomorphism in the middle follows from the projective bundle formula, noting that $\wt X_0$ is a $\mathbb{P}^1$-bundle over $J\wt C_0$. Notice also that, by the above choice of the coordinate $t$, $H^{k}_{lim}(JC_t;\R)$ is a split mixed Hodge structure, with
\begin{equation}\label{eq:splitting}
H^{k}_{lim}(JC_t;\R)\cong H^{k-1}(J\wt C_0;\R(-1))\oplus H^{k}(X_0;\R).
\end{equation}

\subsubsection{Central reducible uninodal curve}
We now consider the second degeneration in \ref{degenerations}: Here the central fibre over $0\in \Delta$ is a reducible curve $C_0 = C_1 \cup C_2$ obtained as the union of two irreducible smooth components $C_1$ and $C_2$ of genera $1$ and $2$ respectively, with a separating node ${y_0} = C_1 \cap C_2$. As before, we first consider the VHS provided by $H^1(C_t; \Q)$, for $t\in \Delta^*$. In this case, $N=0$ and the limit MHS has the following weights: 
\begin{gather*}
        M_0(H^1_{lim}(C_t; \Q))=\operatorname{Im} N=\{e\}\\
        M_1(H^1_{lim}(C_t; \Q))=\operatorname{ker} N=H^1(C_0; \Q)=H^1(C_1; \Q)\oplus H^1(C_2; \Q). 
    \end{gather*}
Therefore, the equality \begin{align}\label{lim in the reducible case}
    H^k_{\lim}(JC_t;\mathbb{Q}) &=\oplus_{r+s=k}\wedge^rH^1(C_1; \Q)\otimes\wedge^sH^1(C_2; \Q)\notag\\
    &=H^k(JC_2; \Q)\oplus H^1(JC_1; \Q)\otimes H^{k-1}(JC_2; \Q) \oplus H^{k-2}(JC_2; \Q(-1))\\
    &\cong H^{k}(J(C_1)\times J(C_2); \Q),\notag
\end{align}
provides $H^k_{\lim}(JC_t; \mathbb{Q})$ with a pure Hodge structure of weight $k$.

\begin{rmk}\label{rmk:2.2.2}
Since $N=0$, and $H^1_{\lim}(C_t;\Q)$ is already split, we cannot make an explicit choice of a coordinate at this stage.
\end{rmk}

\subsection{Degeneration of Ceresa cycles}\label{ss3.1}
Let $\caC\rightarrow \Delta$ be a family as in \ref{degenerations}, with central fibre $C_0$ irreducible, and with one node $x_0$.
Let $\beta: \Delta^* \to \caC, t\to b_t$ be a section fixing a base point $b_t \in C_t, t\neq 0$, and consider a holomorphic extension to $\Delta$, sending $t=0$ to a point $b\in C_0, b\neq x_0$. For $t\in \Delta^*$, the Abel Jacobi map \eqref{AJacobi smooth} embeds $C_t\hookrightarrow JC_t$ through $x\mapsto x-b_t$, and the Ceresa cycle is defined by \begin{displaymath}
    Z_t:=C_t -C^{-}_t\in \mathcal{Z}^{2}(JC_t).
\end{displaymath} 
In general, this cycle gives a non-trivial element in the codimension two Griffiths group of $JC_{t}$. First, we prove that the family of Ceresa cycles extend to a `Ceresa type' cycle $Z_0$ in $X_0$ and compute the cohomology group of $X_0$ with support in $Z_0$. The results below are probably well known to experts, but we provide proofs for completeness.

\begin{prop}\label{prop:4.4}
The family of cycles $Z_t$ extends to $Z_0\coloneqq C_0-C^{-}_{0}$, and the cohomology with support for the pair $(X_0,Z_0)$ satisfies
\begin{displaymath}
H^{k}_{|Z_0|}(X_0;\Q(2))=
  \begin{cases}
    0, &\text{ if } k<4,\\
    \Q(0)\oplus \Q(0) &\text{ if } k=4.
    \end{cases}
\end{displaymath}
\end{prop}
\begin{proof}
The first statement follows from \eqref{Abel non sep}, which embeds $C_{0}$  in $X_{0}$, and from \eqref{identifiaction sections}, which implies that the $[-1]$ morphism can be lifted to $X_{0}$. Next, for ease of writing, we rename $U=JC_0$, which is the open component of $X_0$, $S=J\wt C_0$, the singular locus, and $Z_{0,U}\coloneqq Z_0\cap U$. We will use the following standard fact: From the definition, the cohomology with compact support is given by $H^{k}_{c}(U;\Q(2))=H^{k}(X_0,S;\Q(2))$. Now, we specialize the long exact sequence of relative cohomology with support in $|Z_0|$

\begin{multline*}
\cdots \to H^{k-1}_{|Z_0\cap S|}(S;\Q(2))\to H^{k}_{|Z_{0,U}|}(U;\Q(2))\to H^{k}_{|Z_0|}(X_0;\Q(2))\\
\to H^{k}_{|Z_0\cap S|}(S;\Q(2))\to H^{k+1}_{|Z_{0,U}|}(U;\Q(2))\to\cdots,
\end{multline*}
to $k<4$. Since $S$ is smooth and projective and $|Z_0\cap S|$ is finite, we can use semi-purity to conclude that $H^{k-1}_{|Z_0\cap S|}(S;\Q(2))=H^{k}_{|Z_0\cap S|}(S;\Q(2))=0$. Since $U$ is smooth, semi-purity once again asserts that $H^{k}_{|Z_{0,U}|}(U;\Q(2))=0$. Hence, we conclude $H^{k}_{|Z_0|}(X_0;\Q(2))=0$ for $k<4$. For $k=4$, we will take a slightly different approach. Let $\nu\colon \wt X_0\to X_0$ be the normalization morphism, and $\wt Z_0$ be the strict transform of $Z_0$. Then we show that
\begin{displaymath}
\nu^{\ast}\colon H^{4}_{|Z_0|}(X_0;\Q(2))\to H^{4}_{|\wt Z_0|}(\wt X_0;\Q(2)),
\end{displaymath}
is an isomorphism. We use a variant of the long exact sequence above, for $k=4$, to get
\begin{displaymath}
\cdots\to H^{3}_{|Z_0\cap S|}(S;\Q(2))\to H^{4}_{|Z_0|}(X_0;\Q(2))\xrightarrow{\nu^{\ast}} H^{4}_{|\wt Z_0|}(\wt X_0;\Q(2))\xrightarrow{\rho} H^{4}_{|Z_0\cap S|}(S;\Q(2))\to\cdots .
\end{displaymath}
From purity, we note that $H^{3}_{|Z_0\cap S|}(S;\Q(2))=0$. Hence $\nu^{\ast}$ is injective. We describe the morphism $\rho\colon H^{4}_{|\wt Z_0|}(\wt X_0;\Q(2))\to H^{4}_{|Z_0\cap S|}(S;\Q(2))$ and show that it is zero in this case. Recall that $Y_0$ and $Y_{\infty}$ are identified as $(D,0)\in Y_0\sim (D+x-y,\infty)\in Y_{\infty}$ and $\nu^{-1}(S)=Y_0\coprod Y_{\infty}$ to obtain $X_0$ from $\wt X_0$. It is easy to see that $|C_0\cap S|=\{(x-b,0)\}$ and $|C^{-}_0\cap S|=\{(b-x,\infty)\}$. Further notice that $H^{4}_{|\wt Z_0|}(\wt X_0;\Q(2))=\Q\{\wt C_0\}\oplus \Q\{\wt C^{-}_0\}$. So, in order to define $\rho$, it is enough to do so on $\{\wt C_0\}$ and $\{\wt C^{-}_0\}$. From construction, it follows that $\rho(\{\wt C_0\})=(\wt C_0\cdot Y_{0}-\wt C_0\cdot Y_{\infty})\{(x-b,0)\}$. Since $\wt C_0$ intersects $Y_0$ and $Y_{\infty}$ at points above $(x-b,0)$, the intersection multiplicites $\wt C_0\cdot Y_{0}=\wt C_0\cdot Y_{\infty}$. This implies that $\rho(\{\wt C_0\})=0$. Similar argument shows that $\rho(\{\wt C^{-}_0\})=0$, implying that $\rho\colon H^{4}_{|\wt Z_0|}(\wt X_0;\Q(2))\to H^{4}_{|Z_0\cap S|}(S;\Q(2))$ itself is zero. Hence, we get that $\nu^{\ast}$ is an isomorphism and $H^{4}_{|Z_0|}(X_0;\Q(2))\cong \Q(0)\oplus \Q(0)$, generated by $\{C_0\}$ and $\{C^{-}_{0}\}$.
\end{proof}

Next, we show that $Z_0$ is homologous to zero in the (singular) cohomology of $X_0$.

\begin{prop}\label{prop:4-1}
    The cycle $Z_0$ is homologous to zero in $H^{4}(X_{0};\Q(2))$. Moreover,
    \begin{displaymath}
    \ker\Big(H^{4}_{|Z_0|}(X_0;\Q(2))\to H^{4}(X_0;\Q(2))\Big)\cong \Q(0),
    \end{displaymath}
    is generated by the fundamental class $\{Z_0\}$.
\end{prop}
\begin{proof}
    As noted in the Proposition \ref{prop:4.4}, $H^{4}_{|Z_0|}(X_0;\Q(2))\cong \Q(0)\oplus \Q(0)$. The standard forgetting-the-support morphism $H^{4}_{|Z_0|}(X_0;\Q(2))\to H^{4}(X_0;\Q(2))$ defines the cycle class map $\cl_0$, which lands in the weight zero component of $H^{4}(X_0;\Q(2))$. By equation \eqref{eq:1.2} the weight zero component is given by $H^{4}(\wt X_0;\Q(2))$. Further, we see that $[-1]_{\ast}Z_0=-Z_0$. On the other hand, the involution $[-1]$ acts as identity on $H^{4}(\wt X_0;\Q(2))$. Hence we conclude that $\cl_{0}(Z_0)=0$.

For the last statement, notice that $\cl_{0}(Z_0)=0$ implies that $\cl_0(C_0)=\cl_0(C^{-}_0)$ in $H^{4}(X_0;\Q(2))$. Notice that $\cl_0(C_0)\neq 0$. Indeed, using the normalization $\nu\colon \wt X_0\to X_0$, it is easily seen that $\nu^{\ast}(\cl_0(C_0))=\cl(\wt C_0)\in H^{4}(\wt X_0)$, which is clearly non-zero. Now, every element in $H^{4}_{|Z_0|}(X_0;\Q(2))$ is of the form $\xi=x\{C_0\}+y\{C^{-}_0\}$ for $x,y\in \Q$, where $\{C_0\}$ denotes the fundamental class of a cycle $C_0$ in $H^{4}_{|C_0|}(X_0;\Q(2))$. The morphism $H^{4}_{|Z_0|}(X_0;\Q(2))\to H^{4}(X_0;\Q)$ is defined by the assignment $\{C_0\}\to \cl_0(C_0)$ (respectively $\{C^{-}_0\}\to \cl_0(C^{-}_0)$). Hence we can conclude using $\cl_0(C_0)=\cl_0(C^{-}_0)$ that $x+y=0$, and
\begin{displaymath}
\ker\Big(H^{4}_{|Z_0|}(X_0;\Q(2))\to H^{4}(X_0;\Q(2))\Big)=\Q(0)\{Z_0\}\cong \Q(0).
\end{displaymath}
\end{proof}

\begin{rmk}\label{cicles on norm}
    In the normalization $\widetilde{X}_{0}$ of $X_0$, the class of the cycle $\widetilde{Z}_{0}$ (the normalization of $Z_0$) has an explicit description using the projective bundle formula. Indeed, recall that $\widetilde{X}_{0}=\P(\caO_{J\widetilde{C}_0}\oplus L_{x-y})$ is the projective bundle over $J\widetilde{C}_{0}$. Under the projective bundle isomorphism
\begin{displaymath}
\CH^{2}(\widetilde{X}_{0})\cong \CH^{2}(J\widetilde{C}_{0})\oplus \CH^{1}(J\widetilde{C}_{0}),
\end{displaymath}
the image of the cycle $\widetilde{Z}_{0}$ is given by $(x-y, \widetilde{C}_{0}-\wt {C}^{-}_{0})$, where we view $x-y$ as a zero cycle in $J\wt C_0$. Hence $\wt Z_0$ is homologically trivial in $\mathcal{Z}^{2}(\wt X_0)$. The same applies to $\widetilde{W}_0$, the normalization of $W_0$ in $\widetilde{X}_0$.
\end{rmk}

\subsection{Degeneration of height pairings (irreducible case)}
We now consider the geometric biextension recalled in section \ref{sec 3} attached to Ceresa cycles for the family $\mathcal{C}\to \Delta$ of genus three curves, whose central fibre is uninodal and irreducible, and study the limit height. Notice that, if we embed a generic fibre $C_{t}$ using two effective divisors (or points) $b_{1,t}$ and $b_{2,t}$, to obtain a pair of Ceresa cycles $(Z_t,W_t)$, they will never be disjoint, since the identity element of $JC_t$ will always be in the intersection. For this reason, we fix two sections $\beta_i:\Delta^*\to\Pic^1{\caC}$, $i=1,2$, with $b_{1,t}:=\beta_1(t)\neq\beta_2(t)=:b_{2,t}$, non-effective degree 1 divisors, and for $ t\neq 0$ we let $C_{t, b_{i,t}}$ be the image of the Abel-Jacobi embeddings. Hence,  
\begin{equation}\label{ceresas smooth}
    Z_t:= C_{t, b_{1,t}}-C^{-}_{t, b_{1,t}}, \quad W_t:= C_{t, b_{2,t}}-C^{-}_{t, b_{2,t}}, t\neq 0
\end{equation}
define the corresponding Ceresa cycles. Since the total space $\caC$ smooth, both sections can be extended to the compactified Picard $\overline{\Pic^1(\caC)} \to \Delta$ with locally free limits $b_{1,0}:=\beta_1(0)$ and $b_{2,0}:=\beta_2(0)$. Assume that $b_{1,0}\neq b_{2,0}$, and use these points to define
\begin{displaymath}
    Z_0 := C_{0,b_{1,0}} - C^{-}_{0,b_{1,0}}, \qquad
W_0 := C_{0,b_{2,0}} - C^{-}_{0,b_{2,0}},
\end{displaymath}
as the extensions in $0\in \Delta$ of the families $Z_t, W_t$ as in Proposition~\ref{prop:4-1}. The following is partially discussed in \cite[Lemma 2.1.7]{KRP}:
 \begin{prop}\label{prop:4.5}
    There exists a choice of sections $\beta_i: \Delta \to \Pic^1(\mathcal{C})$ for $i=1,2$, yielding Ceresa cycles $Z_t$ and $W_t$ in $JC_t$ that are disjoint for all $t \in \Delta^*$. Moreover, this disjointness extends to the central fibre, meaning that $Z_0$ and $W_0$ remain disjoint in $X_0$.
     \begin{proof}
         For $t\in \Delta^{\ast}$, we claim that if
         \begin{equation}\label{notin}
         b_{2,t}\notin \{y-x+b_{1,t}, x, y \in C_t\}\cup \{y+x-b_{1,t}, x, y \in C_t\},
     \end{equation} 
         then $Z_t\cap W_t= \emptyset$. Indeed, assume by contradiction that we have $Z_t\cap W_t\neq \emptyset$. Then either $C_{t, b_{1,t}}\cap C_{t, b_{2,t}}\neq \emptyset$ or $C_{t, b_{1,t}}\cap C_{t, b_{2,t}}^{-}\neq \emptyset$. In the first case $b_{2,t}$ belongs to $\{y-x+b_{1,t}, x, y \in C_t\}$, while in the second to $\{y+x-b_{1,t}, x, y \in C_t\}$. The other intersections $C_{t, b_{1,t}}^-\cap C_{t, b_{2,t}}^-, C_{t, b_{1,t}}^-\cap C_{t, b_{2,t}}  $ give equivalent condition. Therefore, since the locus in \eqref{notin} is of dimension two, and $JC_t$ is three-dimensional, there exists $b_{2,t} \in \Pic^1(C_t)$ satisfying \eqref{notin}, and thus $Z_t\cap W_t= \emptyset$. To guarantee that the same holds at the limit, let us recall that $Z_0$, resp. $W_0$, touches the singular locus of $X_0$ only at the image of the node $x_0$, and at $[-1]x_0$. Using \eqref{identifiaction sections}, these are $(x-b_{1,0}, 0)$ and $(b_{1,0}-x,\infty)$ for $Z_0$, and $(x-b_{2,0}, 0)$ and $(b_{2,0}-x,\infty)$ for $W_0$. These are clearly different points, as soon as $b_{1,0}\neq b_{2,0}$. On the other hand, on $C_0\setminus \{x_0\}$, the Abel-Jacobi map behaves much like the smooth scenario. Therefore, we can use the same argument as before to conclude.
     \end{proof}
 \end{prop}
 \begin{rmk}\label{rmk:3-7.5}
 All the results stated previously using the embedding by points (for example, Proposition \ref{prop:4.4}) hold ad verbatim if one changes from an effective divisor of degree one to an arbitrary divisor of degree one to embed $C_t$ inside $JC_t$, or $C_0$ inside $X_0$.
 \end{rmk}
 
  Before stating our key result, let us recall that, as seen in Remark \ref{cicles on norm}, the cycles $\wt{Z}_0$ and $\wt{W}_0$ (i.e. the normalizations of $Z_0$ and $W_0$ in $\widetilde{X}_0$) intersect properly and are homologous to zero in $\mathcal{Z}^{2}(\wt{X}_0)$. Thus, since $\wt{X}_0$ is smooth and projective of dimension three, we can consider another biextension mixed Hodge structure
 \begin{equation*}
      B_{\widetilde{Z}_0, \widetilde{W}_0}=H^{3}(\wt X_0\setminus \vert \wt Z_0\vert , \vert \wt W_0\vert ;\Q(2)),
 \end{equation*}
 with associated height $\Ht(B_{\widetilde{Z}_0, \widetilde{W}_0})$. The rest of the section is devoted to proving our main result. 
 \begin{thm}\label{thm:4-4}
      Let $\mathcal{C}\to \Delta$ be a family of genus three curves with uninodal, irreducible central fibre $C_0$. Let $\caZ=\{Z_{t}\}_{t\in \Delta^{\ast}}$ and $\caW=\{W_{t}\}_{t\in \Delta^{\ast}}$ be two families of Ceresa cycles obtained from $\caC$ by choosing the base points $b_{i, t}, i=1,2$ as above. Then, choosing the coordinate $t$ so that the limiting mixed Hodge structure on $H^1(C_t;\mathbb R)$ and hence on $H^{3}(JC_t;\R)$ is split, we get
\begin{displaymath}    
      \lim_{t\to 0}\Big(\Ht(B_{Z_t,W_t}+\mu\log|t|)\Big)=H(N,F_{\infty}, W)=\Ht(B_{\widetilde{Z}_0, \widetilde{W}_0}).
\end{displaymath}
\end{thm}
Before going into the proof, we make a remark that puts the choice of a coordinate into perspective.
\begin{rmk}\label{rmk:munot0}
It can be deduced from \cite[Theorem 1.1, Proposition 8.3]{HainReed} that $\mu\neq 0$ both when the central fibre is irreducible nodal and when it is reducible. So, limit heights will depend on the choice of coordinates. 
\end{rmk}
Now we get to the proof.
\begin{proof}
The key point is to construct a limit mixed Hodge structure of the variation $\caV_{t}\coloneqq H^{3}(JC_t\setminus \vert Z_t\vert , \vert W_t\vert ;\Q(2))$ of biextension mixed Hodge structures. 
First, let us recall the weight graded pieces of the limit mixed Hodge structure $H^{k}_{lim}(JC_t;\Q)$ for $k=3$, twisting everything by $\Q(2)$ (to bring the weights in the range $[-2,0]$):
\begin{displaymath}
\Gr_{k}^{M}H^{3}_{lim}(JC_t;\mathbb{Q}(2))=
  \begin{cases}
    H^{2}(J\widetilde C_0;\Q(1) ), &\text{ if }k=0,\\
    H^{3}(\widetilde{X}_0;\Q (2)),&\text{ if }k=-1,\\
    H^{2}(J\widetilde C_0;\Q(2))&\text{ if }k=-2,\\
    0,&\text{ otherwise}.
    \end{cases}
\end{displaymath}
To obtain the limit of the biextension variation $B_{Z_t, W_t}$, we need to understand the limit of the different cohomology groups appearing in diagram \ref{fig:height_diagram}. We first need to compute the limit of the Abel-Jacobi image of $Z_t$, resp. $W_t$. 
Let us recall that the Abel-Jacobi image of $Z_t$ is given by the (twisted by $\Q(2)$) extension class
\begin{equation*}
    0 \ra H^3(JC_t, \mathbb{Q}(2))\ra H^{3}(JC_t\setminus \vert Z_t\vert ;\Q(2))\ra \Q(0)\ra 0.
\end{equation*}
Since the limit mixed Hodge structure is functorial with respect to morphisms of admissible variations of mixed Hodge structures (see \cite[Lemma 5.26]{SZ:vmHsI} for the functoriality with respect to the limit Hodge filtration, while the one for monodromy weight filtration is a consequence of its existence), we obtain 
\begin{equation}\label{eq:5.4}
    0 \ra H^3_{lim}(JC_t; \mathbb{Q}(2))\ra H^{3}_{lim}(JC_t\setminus \vert Z_t\vert ;\Q(2))\ra \Q(0)\ra 0,  
\end{equation}
which yields 
\begin{displaymath}
\Gr_{k}^{M}H^{3}_{lim}(JC_t\setminus \vert Z_t\vert ;\Q(2))=
\begin{cases}
    \Q(0)\oplus H^{2}(J\widetilde C_0;\Q(1) ),&\text{ if }k=0,\\
    H^{3}(\widetilde{X}_0;\Q (2)),&\text{ if }k=-1,\\
    H^{2}(J\widetilde C_0;\Q(2)),&\text{ if }k=-2,\\
    0,&\text{ otherwise}.
  \end{cases}
\end{displaymath}

We now identify the $\Q(0)$ term above. Observe that the local system $H^{4}_{|Z_t|}(JC_t;\Q(2))$ on $\Delta^{\ast}$ is a constant variation $\Q(0)\oplus \Q(0)$, generated by $C_t$ and $C^{-}_{t}$ with trivial monodromy. So its limit mixed Hodge structure is $\Q(0)\oplus \Q(0)$ independent of the choice of a coordinate. The cycle class map $\cl_t\colon H^{4}_{|Z_t|}(JC_t;\Q(2))\to H^{4}(JC_t;\Q(2))$ is a morphism of admissible variation with the kernel given by the constant variation $\Q(0)$ generated by $\{Z_t\}$. Since taking limits is an exact functor on admissible variations, we obtain
\begin{align*}
\Q(0)\{Z_t\}&=\ker(\cl_{t})_{\lim}\cong \ker\Big(H^{4}_{|Z_t|}(JC_t;\Q(2))_{\lim}\to H^{4}_{\lim}(JC_t;\Q(2))\Big).
\end{align*}
From Theorem \ref{thm:3.1}, one has the specialization morphism $\text{sp}\colon H^{4}(X_0;\Q(2))\to H^{4}_{\lim}(JC_t;\Q(2))$. Define
\begin{displaymath}
\sigma\colon H^{4}_{|Z_0|}(X_0;\Q(2))\to H^{4}_{|Z_t|}(JC_t;\Q(2))
\end{displaymath}
by sending $\{C_0\}\mapsto \{C_t\}$ and $\{C^{-}_{0}\}\mapsto \{C^{-}_{t}\}$. Easy to see that $\sigma$ is an isomorphism of mixed Hodge structures. The diagram
\[\xymatrix{
        H^{4}_{|Z_0|}(X_0;\Q(2))\ar[r]^{\cl_0}\ar[d]_{\sigma} & H^{4}(X_0; \Q(2))\ar[d]^{\text{sp}}\\
        H^{4}_{|Z_t|}(JC_t;\Q(2))_{\lim}\ar[r]^{\cl_{\lim}} & H^{4}_{\lim}(JC_t;\Q(2))}\]
commutes since $\text{sp}(\cl(C_0))=\cl(C_t)$ and $\text{sp}(\cl(C^{-}_0))=\cl(C^{-}_t)$ respectively. By Proposition \ref{prop:4-1}, we know that $\ker(\cl_0)=\Q(0)\{Z_0\}$. We also have that $\sigma(\{Z_0\})=\{Z_t\}$ generates $\ker(\cl_{\lim})$. Hence, $\sigma$ restricts to an isomorphism of these kernels compatible with the above commutative square. Finally, using the normalization morphism $\nu\colon \wt X_0\to X_0$ we have the commutative diagram
\[\xymatrix{
        H^{4}_{|Z_0|}(X_0;\Q(2))\ar[r]^{\cl_0}\ar[d]^{\cong}_{\nu^{\ast}} & H^{4}(X_0; \Q(2))\ar[d]^{\nu^{\ast}}\\
        H^{4}_{|\wt Z_0|}(\wt X_0;\Q(2))\ar[r]^{\wt{\cl}_0} & H^{4}(\wt X_0;\Q(2)).}\]
Summing up, we have canonical isomorphisms of mixed Hodge structures
\begin{align}\label{isomorphisms Q0}
\Q(0)&=\ker\Big(H^{4}_{|Z_t|}(J(C_t);\Q(2))\to H^{4}(J(C_t);\Q(2))\Big)_{\lim}=\Q(0)\{Z_t\}\\
&\cong \ker\Big(H^{4}_{|Z_t|}(J(C_t);\Q(2))_{\lim}\to H^{4}_{\lim}(J(C_t);\Q(2))\Big)\notag\\
&\cong \ker\Big(H^{4}_{|Z_0|}(X_0;\Q(2))\to H^{4}(X_0; \Q(2))\Big)=\Q(0)\{Z_0\}\notag\\
&\cong \ker\Big(H^{4}_{|\wt Z_0|}(\wt X_0;\Q(2))\to H^{4}(\wt X_0;\Q(2))\Big)=\Q(0)\{\wt Z_0)\}.\notag
\end{align}

Now, invoking the functoriality of limit mixed Hodge structures once again, we know that the duality isomorphism $H^{3}(JC_t\setminus \vert W_t\vert;\Q(2))^{\vee}\cong H^{3}(JC_t, \vert W_t\vert;\Q(2))$ carries over to the limit. Thus, we obtain the limit of the dual Abel-Jacobi image of $W_t$.
\begin{equation}\label{eq:5.5}
0\to \Q(1)\to H^{3}_{\lim}(JC_t, \vert W_t\vert ;\Q(2))\to H^3_{\lim}(JC_t;\mathbb{Q}(2))\to 0.
\end{equation}
Reasoning as before for $\Q(0)$, we obtain that the $\Q(1)$ in \eqref{eq:5.5} is generated by the dual of $\{W_0\}$ (or dual of $\{\wt W_0\}$ using $\nu$). Finally, combining equations \eqref{eq:5.4} and \eqref{eq:5.5}, we obtain the limit MHS for the biextension variation $B_{Z_t, W_t}=H^{3}(JC_t \smallsetminus \vert Z_t\vert , \vert W_t\vert ; \Q(2))$. Concretely, we get the limit version of diagram \ref{fig:height_diagram}:
\[\xymatrix{& 0 & 0 & 0\\
        0\ar[r] & H^3_{\lim}(JC_t; \Q(2))\ar[u]\ar[r]
        & H^{3}_{\lim}(JC_t\setminus\vert Z_t\vert ;\Q(2))\ar[u]\ar[r]
        & \Q(0)\ar[u]\ar[r] & 0\\ 
        0\ar[r] & H^{3}_{\lim}(JC_t,\vert W_t\vert ;\Q(2))\ar[r] \ar[u]
        & H^{3}_{\lim}(JC_t\setminus\vert Z_t\vert ,\vert W_t\vert ;\Q(2)) \ar[r] \ar[u] & \Q(0)\ar[u]^{=}\ar[r] & 0\\
        0\ar[r] & \Q(1)\ar[r]^{=}\ar[u] & \Q(1)\ar[u]\\
        & 0\ar[u] & 0\ar[u]}   \]

Recall the notation $\caV_{t}=H^{3}(JC_t\setminus\vert Z_t\vert ,\vert W_t\vert ;\Q(2))$. From the above diagram and the functoriality of mixed Hodge structures, it is easy to conclude that
\begin{displaymath}
\Gr_{k}^{M}\caV_{\lim}=
  \begin{cases}
    \Q(0)\oplus H^{2}(J\widetilde C_0;\Q(1) ), &\text{ if }k=0,\\
    H^{3}(\widetilde{X}_0;\Q (2)), &\text{ if }k=-1,\\
    \Q(1)\oplus H^{2}(J\widetilde C_0;\Q(2)), &\text{ if }k=-2,\\
    0,&\text{ otherwise}.
    \end{cases}
\end{displaymath}
Notice that in the above description, only the $\Q(0)$ and $\Q(1)$ are the canonical sub-objects.

Let $\delta_{M}$ be the Deligne delta associated with this limit mixed Hodge structure. From the shape of the limit mixed Hodge structure, $\delta_{M}\colon \Gr^{M}_{0}\caV_{\lim}\otimes \R \to \Gr^{M}_{-2}\caV_{\lim}\otimes \R$ with respect to the grading $Y_{(F_{\infty},M)}$. Now, let $Y=Y(N,F_{\infty},M)$ be the unique grading of the original filtration $W$, as described in subsection \ref{ss:2.3}. Since $\delta^{3}_{M}=0$, the component $\delta_{M,-2}$ modulo $\text{ad}(Y)$ is given by the component of $\delta_M$ that goes from $\underbrace{I^{0,0}_{(F_{\infty},M)}}_{Y~\text{weight}~0}$ to $\underbrace{I^{-1,-1}_{(F_{\infty}, M)}}_{Y~\text{weight}~--2}$ (cf. the discussion following Remark \ref{rmk:4.8}), while $\delta_{M,-1}=\delta^{A}_{M,-1}+\delta^{B}_{M,-1}$, where
\begin{align*}
&\delta^{A}_{M,-1}\colon \R(0)\to H^{2}(J\widetilde C_0;\R(2)),\\
& \delta^{B}_{M,-1}\colon H^{2}(J\widetilde C_0;\R(1))\to \R(1). 
\end{align*}
Finally, for analogous reasons, $\delta_{M,0}$ is given only by
\begin{displaymath}
\delta_{M,0}\colon H^{2}(J\widetilde C_0;\R(1))\to H^{2}(J\widetilde C_0;\R(2)),
\end{displaymath}
which, by the choice of coordinates, is trivial. Moreover, by equation \eqref{eqref:3.2-1}, $\delta_{M,-2}$ defines the limit height $H(N,F_{\infty},W)$.

Next, let us recall that, by Remark \ref{cicles on norm}, $\widetilde Z_0,$ is also homologically trivial, hence the Abel-Jacobi class $ \operatorname{AJ}(\widetilde Z_0)\in JH^3(\widetilde X_0;\Q(2))$ is the extension class
\begin{equation}
   0 \to H^3(\widetilde X_0, \Q(2)) \to H^3(\widetilde X_0\setminus \vert \widetilde Z_0\vert , \Q(2))\to\Q(0)\to 0.
\end{equation} 
The same holds for $\widetilde W_0$ and for its class $\operatorname{AJ}(\widetilde W_0)$ (and its dual). Hence we get the corresponding the biextension MHS $B_{{\wt{Z}_0}, \wt{W}_0}$, and the associated height $\Ht(B_{\widetilde{Z}_0, \widetilde{W}_0})$.

By Proposition \ref{prop:4-1}, the family of Ceresa cycles $Z_t$ extends to $Z_0$, which is homologous to zero in $H^4(X_0;\Q(2))$. Thus, the localization sequence, together with Proposition \ref{prop:4.4}, gives us the Abel Jacobi element $AJ(Z_0)$ as the class of the extension
\begin{equation}\label{AJZ_0}
0\to H^3(X_0; \Q(2)) \to H^3(X_0\smallsetminus \vert Z_0\vert;\Q(2) )\to \Q(0)\to 0.
\end{equation}
Using the functoriality of the graded pieces with respect to morphisms of mixed Hodge structures, and \eqref{eq:1.2}, we get the graded pieces of $H^3(X_0\smallsetminus \vert Z_0\vert;\Q(2) )$:
\begin{displaymath}
\Gr_{k}^{W}H^3(X_0\smallsetminus \vert Z_0\vert;\Q(2))=
  \begin{cases}
    \Q(0),&\text{ if }k=0,\\
    H^{3}(\wt X_0;\Q(2)),&\text{ if }k=-1,\\
    H^{2}(J\wt C_0;\Q(2)),&\text{ if }k=-2,\\
    0,&\text{ otherwise.}
  \end{cases}
\end{displaymath}
Similarly, $AJ(W_0)$ is given by the class of the extension:
\begin{equation}\label{AJW_0}
    0\to H^3(X_0;\Q(2)) \to H^3(X_0\smallsetminus \vert W_0\vert;\Q(2))\to \Q(0)\to 0.
\end{equation}
Using the long exact sequence in relative cohomology and the fact that $H^{3}(|W_0|)=0$ and $\text{coker}(H^{2}(X_0;\Q(2))\to H^{2}(|W_0|;\Q(2)))\cong \Q(1)$, we get
\begin{displaymath}
0\to \Q(1)\to H^{3}(X_0, |W_0|;\Q(2))\to H^{3}(X_0;\Q(2))\to 0.
\end{displaymath}
It follows that combining \eqref{AJZ_0} with the dual of \eqref{AJW_0}, we get the MHS $H^3(X_0 \smallsetminus \vert Z_0\vert , \vert W_0\vert ; \Q(2))$ and a diagram similar to diagram \ref{fig:height_diagram}, with the mixed Hodge structure $H^{3}(X_0;\Q(2))$ on the top left corner. Hence, we get the graded pieces (also using \eqref{eq:1.2})
\begin{equation}\label{Biext Z_0}
\Gr_{k}^{W}H^3(X_0\smallsetminus \vert Z_0\vert, \vert W_0\vert;\Q(2))=
  \begin{cases}
    \Q(0),&\text{ if }k=0,\\
    H^{3}(\wt X_0;\Q(2)),&\text{ if }k=-1,\\
    \Q(1)\oplus H^{2}(J\wt C_0;\Q(2)),&\text{ if }k=-2,\\
    0,&\text{ otherwise.}
  \end{cases}
\end{equation}
We think of $H_{Z_0,W_0}\coloneqq H^3(X_0\smallsetminus \vert Z_0\vert, \vert W_0\vert;\Q(2))$ as a framed mixed Hodge structure, with framing given by
\begin{displaymath}
\phi\colon \Q(0)\to \Gr^{W}_{0}H_{Z_0,W_0},~ 1\mapsto \{Z_0\},
\end{displaymath}
and
\begin{displaymath}
\psi=\psi_{B_{\wt Z_0, \wt W_{0}}}\circ \Gr^{W}_{-2}(\nu^{\ast})\colon \Gr^{W}_{-2}H_{Z_0,W_0}\to \Q(1),
\end{displaymath}
where $\psi_{B_{\wt Z_0, \wt W_{0}}}$ is part of the canonical framing $(\phi_{B_{\wt Z_0, \wt W_{0}}}, \psi_{B_{\wt Z_0, \wt W_{0}}})$ for the biextension MHS $B_{\wt Z_0, \wt W_{0}}$. Notice that, since $\{Z_0\}\mapsto \{\wt Z_0\}$ under $\nu^{\ast}$, we also have $\phi_{B_{\wt Z_0, \wt W_{0}}}=\Gr^{W}_{0}(\nu^{\ast})\circ \phi$, making $\nu^{\ast}$ a morphism of framed mixed Hodge structures. Hence, we get 
\begin{displaymath}
\Ht(B_{\wt Z_{0}, \wt W_0})=2\pi \Ht_{2}(H_{Z_0,W_0}).
\end{displaymath}
We comment here that if we follow the convention of \cite{BGGP:Height}, the two heights should be equal (cf. Proposition \ref{height-ind-morphism}). However, the presence of the factor of $2\pi$ matches the convention of \cite{BP:Jump}, and that of classical archimedean height pairing.

We now compare the Deligne $\delta_M$ with the Deligne splitting associated with \eqref{Biext Z_0} (and with the biextension $B_{\widetilde{Z}_0, \widetilde{W}_0}$). It is enough to tensor everything with $\R$, since the Deligne splitting is a real operator.  First, we need to understand the relation between $H^{3}(X_0\setminus\vert Z_0\vert ;\R(2))$ and $H^{3}_{\lim}(JC_t\setminus\vert Z_t\vert  ;\R(2))$.

Using the inclusion $H^{1}(C_0;\R)\hookrightarrow H^{1}_{\lim}(C_t;\R)$, we obtain the inclusion $H^{3}(X_0;\R(2))\overset{\iota}{\hookrightarrow} H^{3}_{\lim}(JC_t;\R(2))$. Moreover,  by the splitting in \eqref{eq:splitting}, we see that the real limit Abel-Jacobi class $AJ(Z_t)_{lim}\otimes \R$, i.e. the extension $H^3_{lim}(JC_t \smallsetminus \vert Z_t \vert; \R(2))$, can be described as
\begin{gather*}
    AJ(Z_t)_{lim}\otimes \R=(\xi_1, \xi_2) \in Ext^1_{\R-MHS}(\R(0), H^3(X_0; \R(2))\oplus H^2(J\widetilde{C}_0); \R(1)), \\
    \text{where} \quad Ext^1_{\R-MHS}(\R(0), H^3(X_0; \R(2))\oplus H^2(J\widetilde{C}_0); \R(1)) \\
    \cong Ext^1_{\R-MHS}(\R(0), H^3(X_0, \R(2)))\oplus Ext^1_{\R-MHS}(\R(0), H^2(J\widetilde{C}_0); \R(1)).
\end{gather*}
Since $\R(0)$ and $H^{2}(J\wt C_0;\R(1))$ are pure real Hodge structures of the same weight, we get $\xi_2=0$. This, combined with the equality provided by \eqref{isomorphisms Q0}, show that the class $AJ(Z_t)_{lim}\otimes \R$ fits in the following diagram
\[\xymatrix{
        0\ar[r] & H^3(X_0, \R(2))\ar@{^{(}->}[d]\ar[r]
        & H^{3}(X_0\setminus\vert Z_0\vert ;\R(2))\ar@{^{(}->}[d]\ar[r]\ar@{^{(}->}[d]
        & \R(0)\ar[r]\ar[d]^{=} & 0\\ 
        0\ar[r] & H^3(X_0, \R(2)) \oplus H^2(J\widetilde{C}_0); \R(1)) \ar[r] 
        & H^{3}_{\lim}(JC_t\setminus\vert Z_t\vert  ;\R(2)) \ar[r]  & \R(0)\ar[r] & 0,}\]
and it is obtained as the pushout via $\iota$ of $H^3(X_0\setminus\vert Z_0\vert ;\R(2))$. Therefore $H^{3}_{\lim}(JC_t\setminus\vert Z_t\vert  ;\R(2)) \cong H^{3}(X_0\setminus\vert Z_0\vert ;\R(2)) \oplus H^2(J\widetilde{C}_0); \R(1))$ and, by the commutativity of the diagram, if $e_0\in I^{0,0}H^{3}(X_0\setminus\vert Z_0\vert ;\R(2))_{\C}$ is the unique lift of $1\in \R(0)$, then its image under the pushout is $(e_0,0)\in H^{3}_{\lim}(JC_t\setminus\vert Z_t\vert  ;\R(2))$. 

Using (relevant versions of) diagram \ref{fig:height_diagram}, we get an injective morphism of real mixed Hodge structures
\begin{displaymath}
j: H_{Z_0,W_0}\otimes\R\hookrightarrow \caV_{lim}\otimes\R,
\end{displaymath}
such that it restricts to $\iota$ on $H^{3}(X_0;\R(2))$ and is the pushout along $\iota$. We also denote by $e_0$ the (unique) lift of $1\in \R(0)$ to $I^{0,0}(H^{3}(X_0\setminus\vert Z_0\vert, \vert W_0\vert ;\R(2)))_{\C} $. From the properties of Deligne bigrading and limit mixed Hodge structures, we see that
\begin{align*}
& I^{0,0}(\caV_{\lim})=I^{0,0}_{1}\oplus I^{0,0}_{2};~ Y~\text{weights}~0,-1\\
& I^{-1,-1}(\caV_{lim})=I^{-1,-1}_{1}\oplus I^{-1,-1}_{2};~ Y~\text{weights}~-2, -1.
\end{align*}
For the mixed Hodge structure $H_{Z_0,W_0}$, the $(-1,-1)$-component of the Deligne bigrading is of the shape
\begin{displaymath}
J^{-1,-1}(H_{Z_0,W_0})=J^{-1,-1}_{1}\oplus J^{-1,-1}_{2},
\end{displaymath}
where the second component comes from $H^{2}(J\wt C_0;\R(2))\otimes \C$. Following the above discussion, we get that 
\begin{displaymath}
j(e_0)\coloneqq a=(e_0,0)\in I^{0,0}_{1}.
\end{displaymath}
By Lemma \ref{lem:2.5}, we get $\delta_{M}(j(e_0))=j(\delta_{H_{Z_0,W_0}}(e_0))$. Hence, we get
\begin{displaymath}
\delta_{M}(a)=\delta_{M,-2}(a)+\delta^{A}_{M,-1}(a).
\end{displaymath}
Applying $\psi$ to $\delta_{H_{Z_0,W_0}}(e_0)=j^{-1}\delta_{M}(a)$ we get
\begin{displaymath}
2\pi\Ht_2(H_{Z_0,W_0})=H(N,F_{\infty}, W)+\psi(j^{-1}\delta^{A}_{M,-1}(a)).
\end{displaymath}
Now, the way we defined $\psi$ makes it clear that it is a retraction of $\Gr^{W}_{-2}H_{Z_0,W_0}=\Q(1)\oplus H^{2}(J\wt C_0;\Q(2))$ to $\Q(1)$. Using the description of Deligne's theorem on the grading $Y=Y(N,F_{\infty},W)$, it is easy to see that $j(H^{2}(J\wt C_0;\R(2)))\subset E_{-1}(Y)_{\R}$, where $E_{-1}(Y)$ denotes the $-1$-eigenspace for $Y$. Hence $j^{-1}\delta^{A}_{M,-1}(a)\in H^{2}(J\wt C_0;\R(2))$, which implies that $\psi(j^{-1}\delta^{A}_{M,-1}(a))=0$. Putting everything together, we conclude
\begin{displaymath}
H(N,F_{\infty}, W)=2\pi\Ht_2(H_{Z_0,W_0})=\Ht(B_{\wt Z_0,\wt W_0}).
\end{displaymath}
By Corollary \ref{cor:3.2-3}, this height is constant along the particular real sub-orbit (corresponding to $\Im(\lambda)=-c$) of the Nilpotent orbit.

\end{proof}

\subsection{Degeneration of height pairings (reducible case)} We now consider the second degeneration discussed in \ref{itm-1.3.3-2} of \ref{degenerations}: We have a family $\caC\rightarrow \Delta$ of genus three curves, whose central fibre $C_{0}=C_{1}\cup C_{2}$ is reducible, and uninodal, with smooth irreducible components $C_{1}$ of genus $1$ and $C_{2}$ of genus $2$, meeting at $y_0=C_0\cap C_1$. Moreover, let $\beta_i: \Delta^*\to \Pic^1(\caC), $ $i=1,2$, define the base points for the Abel Jacobi embeddings for the Ceresa Cycles $Z_t, W_t, t\in\Delta^*$ as in \eqref{ceresas smooth}, and consider two extensions of the $\beta_i$'s to $0\in \Delta$ such that $\beta_1(0)=(L_1, L_1')\in JC_1\times \Pic^1(C_2)$ and $\beta_2(0)=(L_2, L_2')\in \Pic^1(C_1)\times C_2$. Then, let $C_{0,x}$ and $ C_{0,y}$ in $JC_1\times JC_2$ be the two Abel Jacobi embedding of $C_0$ obtained using \eqref{cycle X} and \eqref{cycle X, punto base su C_1}. Define\begin{displaymath}
     Z_0:= C_{0,\beta_1(0)}-C^{-}_{0,\beta_1(0)}, \quad W_0:= C_{0,\beta_2(0)}-C^{-}_{0,\beta_2(0)}
 \end{displaymath}
 as the extension of the family of the Ceresa cycle. We have the following: 
\begin{prop}\label{prop:5.4}
    The family of Ceresa cycles $Z_t\in \mathcal{Z}^2(JC_t)$ degenerates to a cycle ${Z}_0:=C_0-C_0^-$ in $\mathcal{Z}^2(JC_0)=\mathcal{Z}^2(JC_1\times JC_2)$ which is homologous to zero. Furthermore, the Ceresa cycles $Z_t$ and $W_t$ are disjoint in $JC_t$, and the same conclusion holds for $Z_0$ and $W_0$ in $JC_0\cong JC_1\times JC_2$.
\end{prop}
\begin{proof}
The image of $C_0\subset JC_0$ via the Abel-Jacobi map \eqref{cycle X} is given by two irreducible components meeting at the image of the node. Let us call the two components $C_1, C_2$.  Furthermore, let us observe that, since $JC_0\cong JC_1\times JC_2$, the morphism $[-1]$ of $JC_{0}$ corresponds to the multiplication $([-1],[-1])$ on $JC_1\times JC_2$. Let $q_i\in C_i, i=1,2$ be the preimages of the node $y_0\in C_0$, and, to simplify the argument, let us assume that the Abel-Jacobi embedding is defined using as base point a smooth Weierstrass point $x\in C_2$.  Hence, \begin{align*}
    Z_0&= \caO_{C_1}\times (C_2-C^{-}_2)+C_1\times\caO_{C_2}(q_2-x)-C^{-}_1\times\caO_{C_2}(x-q_2)
\end{align*}
Since $JC_2$ is an abelian surface, $C_1\times\caO_{C_2}(q_2-x)$ and $C_1\times\caO_{C_2}(x-q_2)$ are homotopic, using, as always, that $[-1] $ acts as the identity on even homology, we can conclude the result. A similar argument applies when $x$ is not a Weierstrass point, and also when $x$ is a non-effective degree 1 divisor. 

Concerning disjointness of $Z_t$ and $W_t$: for $t\neq 0$, the proof is the same as in Proposition \ref{prop:4.5}, while for $t=0$ it follows from the choice of the base points for the Abel-Jacobi embeddings.  
\end{proof} 
Then, as before, we can consider the geometric biextension attached to the Ceresa cycles $Z_t, W_t$, and study its height in the limit. Contrary to what happens in the uninodal, irreducible case, here the limit Jacobian $X_0$ is already smooth and projective, and no normalization is necessary. Let $B_{Z_0, W_0}$ be the biextension mixed Hodge structure constructed using $Z_0, W_0$.  Our result in this setting is the following.
\begin{thm}\label{thm:3-8}
     Let $\mathcal{C}\to \Delta$ be a family of genus three curves with uninodal, reducible central fibre $C_0=C_1\cup C_2$. Let $\caZ=\{Z_{t}\}_{t\in \Delta^{\ast}}$ and $\caW=\{W_{t}\}_{t\in \Delta^{\ast}}$ be two families of Ceresa cycles obtained from $\caC$ by choosing the base points $b_{i, t}, i=1,2$ as above. Then there exists a choice of the coordinate $t$ such that $$H(N, F_\infty, W)=\Ht(B_{Z_0, W_0}).$$
\end{thm}
\begin{proof}
As before, denote by $\caV_{t}$ the biextension variation $H^{3}(JC_t\setminus |Z_t|, |W_t|;\Q(2))$. The proof follows the same ideas as Theorem \ref{thm:4-4}. Here, the limit mixed Hodge structure $H^{1}_{\lim}(C_t;\Q(1))\cong H^{1}(C_1;\Q(1))\oplus H^{1}(C_2;\Q(1))$ is actually a pure Hodge structure for any choice of coordinates, and the same holds for $H^3_{\lim}(JC_t;\Q(2))\cong H^3(X_0; \Q(2))$, as shown in equation \eqref{lim in the reducible case}. The nilpotent monodromy $N$ of the biextension variation $\caV_t$ is trivial on $\Gr^{W}_{-1}\caV_t$. It follows that the limit Abel-Jacobi classes of $Z_t$ and $W_t$ are isomorphic to $AJ(Z_0)$ and to $AJ(W_0)$ (see the discussion on pp. 299-300 of \cite{GGK:10}). However, $N$ itself is non-trivial, and therefore, a priori, the limit biextension $\caV_{\lim}$ and $B_{Z_0, W_0}$ are two biextensions that determine the same point $(AJ(Z_0), AJ(W_0))\in JH\times JH^\vee$, where $H:=H^3(X_0;\Q(2))$. Passing to real coefficients and identifying $\caB_{H}\otimes \R\cong \R$ as in \ref{sec 3}, it is obvious that there exists a $c\in \R$ such that 
\begin{displaymath}
    \Ht(B_{Z_0,W_0})-\Ht(\caV_{\lim, t})=c.
\end{displaymath}
Therefore, taking $a\in \C$ such that $\log |a|=c/\mu$ (notice that $\mu\neq 0$ following Remark \ref{rmk:munot0}), we can change the coordinate $t\in \Delta^*$ to $s=at$ so that, as observed in Remark \ref{rmk:on the change of coordinate}, we obtain\begin{displaymath}
    \Ht(\caV_{\lim, s})=\Ht(\caV_{\lim, t})+\mu\log |a|=\Ht(B_{Z_0,W_0}).
\end{displaymath}
In other words, there exists a choice of the coordinate yielding the claimed equality between the heights. 
    
\end{proof}

\bibliographystyle{amsalpha}

\end{document}